\documentclass[11pt]{amsart}
\usepackage{graphics}
\usepackage[english]{babel}
\usepackage{amssymb,amsmath,amsfonts,mathrsfs}
\usepackage{datetime}
\usepackage{color}
\usepackage{units}
\usepackage{ulem}
\usepackage{lipsum}
\usepackage{amsfonts}
\usepackage{enumerate}
\usepackage{cite}
\usepackage{tikz}
\usetikzlibrary{matrix}
\usepackage[pdftex]{hyperref}
\RequirePackage{amscd}
\RequirePackage{epic}
\RequirePackage[all]{xy}
\RequirePackage{url}
\RequirePackage[shortlabels]{enumitem}
\usepackage{empheq}
\usepackage{epsfig}
 \usepackage{perpage}
\MakePerPage{linenumbers}
\usepackage{lipsum}
\usepackage[english]{babel}

\usepackage[utf8]{inputenc}
\usepackage{cite}
\RequirePackage{amscd}
\RequirePackage{epic}
\RequirePackage{eepic}
\RequirePackage[all]{xy}
\RequirePackage{url}
\RequirePackage[shortlabels]{enumitem}
\usepackage{empheq}
\usepackage{nomencl}
\usepackage{epsfig}
\usepackage{graphicx}

\newcommand{\comment}[1]{}
   
    \newcommand{\set}[1]{{\left\{#1\right\}}}
\newcommand{\pa}[1]{{\left(#1\right)}}
\newcommand{\sq}[1]{{\left[#1\right]}}

\newcommand{\abs}[1]{{\left|#1\right|}}
\newcommand{\norm}[1]{{\left |#1\right |}}

\newcommand{\T}{\mathbb{T}}
\newcommand{\Z}{\mathbb{Z}}
\newcommand{\R}{\mathbb{R}}
\newcommand{\C}{\mathbb{C}}
\newcommand{\teta}{\theta}

\newcommand{\eps}{\varepsilon}

\newcommand{\na}{\widehat{n}}

\newcommand{\base}[1]{{\frac{\partial}{\partial{#1}}}} 

\newcommand{\gr}[1]{\textbf{#1}}

\newcommand{\id}{\operatorname{Id}}

\newcommand{\ad}{\operatorname{ad}}

\newcommand{\ri}{r}
\usepackage{amsthm}

\RequirePackage{url}

    \newtheorem*{thm*}{Theorem}
    \newtheorem*{cor*}{Corollary}
    \newtheorem*{teo normal}{"Twisted Conjugacy" Theorem}
    \newtheorem{ass}{Assumption}

\newtheorem{Thm}{Theorem}
\newtheorem{prop}{Proposition}[section]
\newtheorem{cor}[prop]{Corollary}
\newtheorem{lemma}[prop]{Lemma}
\newtheorem{remark}[prop]{Remark}

\newtheorem{defn}[prop]{Definition}

\newtheorem{exam}{Example}

\numberwithin{equation}{section}

\newcommand{\g}{\gamma}

\newcommand{\s}{{\sigma}}

\newcommand{\rf}{{r'}}

\newcommand{\papa}[1]{\frac{\partial}{\partial x_{#1}}}

\newcommand{\N}{{\mathbb N}}
\newcommand{\Q}{{\mathbb Q}}

\newcommand{\cC}{{\mathcal C}}

\newcommand{\cG}{{\mathcal G}}
\newcommand{\cH}{{\mathcal H}}
\newcommand{\cI}{{\mathcal I}}
\newcommand{\cJ}{{\mathcal J}}
\newcommand{\cK}{{\mathcal K}}

\newcommand{\cM}{{\mathcal M}}
\newcommand{\cN}{{\mathcal N}}

\newcommand{\cR}{{\mathcal R}}

\newcommand{\cT}{{\mathcal T}}

\newcommand{\cV}{{\mathcal V}}

\newcommand{\cY}{{\mathcal Y}}

\newcommand{\fm}{{\mathfrak{m}}}

\newcommand{\tc}{{\mathtt{c}}}
\newcommand{\td}{{\mathtt{d}}}

\newcommand{\tf}{{\mathtt{f}}}
\newcommand{\tg}{{\mathtt{g}}}

\newcommand{\tm}{{\mathtt{m}}}

\newcommand{\tC}{{\mathtt{C}}}
\newcommand{\tD}{{\mathtt{D}}}

\newcommand{\tK}{{\mathtt{K}}}

\newcommand{\tM}{{\mathtt{M}}}

\newcommand{\be}{{\bf e}}

\usepackage{bm}

\newcommand{\al}{{\alpha}}

\newcommand\norma[1]{\left\lVert#1\right\rVert}

\newcommand{\im}{{\rm i}}
\newcommand{\jap}[1]{\langle #1 \rangle}
\newcommand{\und}[1]{\underline{#1}}

\newcommand{\e}{{\varepsilon}}

\newcommand{\diag}{{\rm diag}}

\newcommand{\fin}{{\rm fin}}

\definecolor{aquamarine}{rgb}{0,0.5,0.5}

\newcommand{\bnorm}[1]{{|\mkern-6mu |\mkern-6mu | \,  #1 \,  |\mkern-6mu |\mkern-6mu |}  }

\newcommand{\nnorm}[1]{{\left\vert\kern-0.25ex\left\vert\kern-0.25ex\left\vert #1 
    \right\vert\kern-0.25ex\right\vert\kern-0.25ex\right\vert}}

\newcommand{\es}{e^{\sq{S,\cdot}}}

\usepackage{framed,enumitem} 

\newcommand{\crac}{{\mathtt C'}}

\newcommand{\iacopo}{\cI^{(1)}}
\newcommand{\ozio}{\cI^{(0)}}
\newcommand{\dom}{\cI^{(2)}}

\newcommand{\jacopo}{\cJ^{(1)}}

\newcommand{\jdom}{\cJ^{(2)}}

\begin{document}

\author{Jessica Elisa Massetti}
\address{Università degli Studi Roma ``Tor Vergata"}
\email{massetti@mat.uniroma2.it}

\author{Michela Procesi}
\address{Università degli Studi Roma Tre}
\email{procesi@mat.uniroma3.it}

\author{Laurent Stolovitch}
\address{Laboratoire J.A. Dieudonné, Université Côte d'Azur}
\email{Laurent.STOLOVITCH@univ-cotedazur.fr}
 
 \title[Resonant normal form for infinite dimensional vector fields]{{Invariant sets through} resonant normal form for infinite dimensional holomorphic vector fields}
\begin{abstract}
{In this paper, we study infinite dimensional holomorphic vector fields on sequence spaces, having a fixed point at $0$. Under suitable hypotheses we prove the existence of analytic invariant submanifolds passing through the fixed point. The restricted dynamics is analytically conjugate to the linear one under some Diophantine-like condition.}
\end{abstract} 
\maketitle
\setcounter{tocdepth}{2} 
\tableofcontents
\section{Introduction} 


In this paper we shall {prove the existence of analytic invariant submanifolds passing through a fixed point of analytic vector fields in infinite dimension. The restricted dynamics is analytically conjugate to the linear one. These invariant sets are obtained by extending to the infinite dimensional setting the notion of normal forms } of holomophic vector fields first introduced by Poincar\'e and Dulac in the $19$th century. To this purpose, let us consider a sequence space indexed by some countable index set $I$, 
with variables $x=\pa{x_k}_{k\in I}$. As it is habit, we introduce formal power series and formal vector fields, i.e.
\begin{equation}
\label{notation0}
f(x)=\sum_{q\in \N^I_{\fin}} f_q x^q\,, \quad \quad  V(x) = \sum_{k\in I, q\in \N^I_{\fin}} V_q^{(k)} x^q \frac{\partial}{\partial x_k}
\end{equation}
where  

$$\N^I_{\fin} := \set{q\in\N^I\, : \, \|q\|_{\ell^1}:=\sum_{i\in I}q_i < \infty },$$
is the set of elements in $\N^I$ with finite support.

If the set $I$ is not finite, in general we cannot expect that the objects above behave well under products or commutators for instance, so that even at a formal level, some hypothesis are needed in order to perform normal form techniques.

In order to minimize technical questions let us start by considering the finite dimensional case, that is $|I| < \infty$. In fact, the main ideas and strategy will be applied in the infinite dimensional setting, provided we define an appropriate functional framework. \\

{\textit{The case $|I|<\infty$.}} Let us consider a holomorphic vector field in the neighborhood of the origin in $\C^I$ of the form 
\begin{equation}
\label{finito vf}
X = \tD(\lambda) + P
\end{equation}
where $\tD(\lambda) = \sum_{k\in I}\lambda_k\base{x_k},\,\lambda_k\in\C$ and $P$ is a holomorphic vector field with a zero of order at least two at the origin. 

A very classical question is whether it is possible to conjugate $X$ to its linear part $\tD(\lambda)$. As it is well known this is in general not possible even at a formal level because of the presence of \emph{resonances}. In fact, the Poincar\'e-Dulac normal form procedure shows that \eqref{finito vf} can be formally conjugated at best to a normal form 
$$
Y = \tD(\lambda) + Z, \quad \quad [\tD(\lambda),Z] = 0\,.
$$
The \emph{resonant} term $Z$ is a formal power series of the form 

\begin{equation}
\label{risonante finito}
Z(x) = \sum_{k\in I}\sum_{q\in\N^I} Z_q^{(k)} x^q \frac{\partial}{\partial x_k},\quad\quad  (q\cdot\lambda - \lambda_k) Z^{(k)}_q  = 0\, \quad \forall q, k.
\end{equation}

It is usually not possible to conjugate it to a normal form through an analytic transformation\cite{arnold-nf,Bruno}. One might then wonder whether it is possible to conjugate $X$ to another model, which coincides with $\tD(\lambda)$ only if restricted to some appropriate manifold that is invariant under the linear flow. A natural choice is represented by the zero set of the constants of motion, that is those functions (either holomorphic or formal) which are invariant under  $\tD(\lambda)$. Similarly to \eqref{risonante finito}, such functions are of the form 

\begin{equation}
\label{costante finito}
C(x) = \sum_{q\in\N^I} C_q x^q,\quad \quad (q\cdot\lambda) C_q = 0\,\quad \forall q.
\end{equation}
Consider now the sub-lattice 
$$
{\cM}_\lambda:= \{ Q\in \N^I:   Q\cdot \lambda =0 \}
$$
and let $Q_1,\ldots, Q_n$ be its generators. Then the ring of constants of motions is generated by the elementary monomials $h_i (x):= x^{Q_i}$. In this line of thoughts it is natural to take into consideration the manifold 
$$
\Sigma := \set{x\in\C^I,\, :\,  h_i(x) = 0\, \forall i}\,.
$$
Note that any vector field of the form 
\begin{equation}\label{gigino}
\sum_{k\in I} C^{(k)}(x)x_k \frac{\partial}{\partial x_k},\,
\end{equation}
where $C^{(k)}(x)$ is a constant of motion, is resonant and vanishes on $\Sigma$. On the other hand resonant vector fields might not all be of this form. Indeed, a resonant monomial $x^q\base{x_k}$ might have $q_k = 0$, so that one cannot factorize  $x_k \base{x_k}$ out of it. We shall refer to such   vector fields as \textit{resonant non diagonal}, and denote the monomial vector fields generating them as 
\begin{equation}\label{resonant mono}
 x^{p+e_k}\base{x_k}\,,\quad p\in \Z^I\setminus\N^I \quad \mbox{such that}\;
p+e_k\in \N^I\,,\quad p\cdot\lambda=0.
\end{equation}
By contrast, if $p\in\N^I$, then it is in $\cM_\lambda$ and the vector field is of type \eqref{gigino} which we refer to as \textit{resonant diagonal} vector fields. 
 Of course, any resonant vector field multiplied by a constant of motions is still resonant and vanishes on $\Sigma$. However, if the set of non diagonal resonant vector fields is non empty, then some of them necessarily do not vanish on $\Sigma$ nor are tangent to it. More precisely, there exist a finite list of generators $P_1,\ldots,P_m\in\Z^I\setminus\N^I$ such that any $p$ as in formula \eqref{resonant mono} can be uniquely written as 
\begin{equation}\label{formalina}
p = P_i + Q\,,\quad Q\in\cM_\lambda\,,
\end{equation}
ad of course if $Q = 0$ then the corresponding monomial cannot vanish on $\Sigma$.
 \\ 
We shall denote by $\Delta_\lambda$ the set of those $p\in\Z^I$ such that either $p\in\cM_\lambda$ or $p$ has the form \eqref{formalina}, so that the resonant vector fields are generated by  $x^{p + e_k}\base{x_k}$, with $p\in\Delta_\lambda$.\\

By construction, there exists $\tM^\ast \in\N$ such that all monomial resonant vector fields which have a zero of order $\ge \tM^\ast+1$ are of the form $$ x^{Q_i}x^{Q_j}x^{p + e_k}\base{x_k}$$ for some $i,j$ and $p\in\Delta_\lambda$.
{Let us illustrate our definitions: let us consider a nonlinear perturbation of the vector field $\tD(\lambda):=2x_1\partial_{x_1}+x_2\partial_{x_2}+\zeta\left(x_3\partial_{x_3}-x_4\partial_{x_4}\right)$ for some positive irrational number $\zeta$. So $\mathcal M_\lambda$ is generated by $x_3x_4$, $\Sigma=\{x_3x_4=0\}$, the only non diagonal resonant terms are  generated by $x_2^2\partial_{x_1}$ and $\tM^*=5$.
	
	A formal normal form is of the form $$\tD(\lambda)+cx_2^2\partial_{x_1}+ f_2(x_3x_4)x_2^2\partial_{x_1}+f_3(x_3x_4)x_3\partial_{x_3}+ f_4(x_3x_4)x_4\partial_{x_4},$$ where $c$ is a constant and the $f_i$'s are  formal power series of a single variable, vanishing at the origin. 
	 If $c=0$, then the set $\Sigma$ is invariant by the formal normal form, which reduces to the linear vector field $\tD(\lambda)$ on $\Sigma$. Of course there is no reason why both the transformation or the normal form should be analytic. In finite dimension, it is known that such a "linearization on analytic sets" result holds in the analytic setting if the linear part $\tD({\lambda})$ satisfies a Diophantine-like condition even if there is no convergent transformation to a normal form. This was proved by one of the authors in \cite{Stolo-dulac} by a majorant method. In the non-resonant volume preserving case, $\mathcal M$ is generated by the sole monomial $x_1\cdots x_n$ and the result was obtain by B. Vallet \cite{vallet}.  \\ 
	The aim of this article is to prove the analytic linearization result with a proof based on a Newton's method which is uniform in the dimension and thus well suited for the infinite dimensional case. As a byproduct, in finite dimension this gives a completely new proof of \cite{Stolo-dulac}, under slightly different hypotheses, namely, in \cite{Stolo-dulac}, the restriction of linear part $\tD({\lambda})$ to $\Sigma$ is assumed to satisfy Bruno's condition. This is known to be weaker than the Diophantine condition. On the other hand, in the aforementioned article, the resonances are all assumed to be of diagonal type and this is not assumed in the present article.
	To be completely explicit, our main result, {Theorem \ref{main-thm},} is an infinite dimensional version of the following:}
	
	\medskip
\noindent
\textit{Consider an analytic vector field of the form 
\begin{equation}
\label{campo}
X = \tD(\lambda) + Z + P,
\end{equation}
where $Z$ is a diagonal resonant vector field with a zero of order at least two at the origin while $P$ has a zero of order at least $\tM^\ast + 1$.
Assume moreover that $\lambda$ is Diophantine modulo $\Delta_\lambda$ (see Definition \ref{def-dioph}). There exists a sufficiently small radius $\rho>0$ and a diffeomorphism $\psi$ tangent to the identity holomorphic on the ball $B_\rho(0)$ such that
\begin{equation}\label{normale}
	\psi_* X = \tD(\lambda) + Z + R,
\end{equation}
where $R $ vanishes on $\Sigma$ and  is generated by monomials of the form $x^{Q_i}x^{Q_j}x^{q}\base{x_k}$. 
}
\medskip

\textit{{The case $|I| = \infty$.}} In infinite dimension, as we mentioned before, the problem of normal forms might not even make sense at a formal level.
In order to keep things simple we shall focus on the example where
$I= \Z \times \{+1,-1\}$, which in our opinion contains all the main difficulties without excessively cumbersome notations. 
  In the same spirit
we restrict to vector fields that satisfy some symmetry (e.g. translation invariance in models coming from PDEs) 
  as shown in \cite{stolo-procesi}. This allows to define the notion of formal normal form of vector field and it would allow us to develop a theory of formal Lie algebras of infinite dimensional vector fields, by a straightforward generalization of \cite{stolo-procesi}. However we are more interested in the analytic category, hence we fix from the beginning an appropriate functional setting. More precisely we proceed as follows.
  \begin{itemize}[$\bullet$,leftmargin=*]
  	\item  We choose as functional space
  \[
  \tg_s= \tg_s(I,\C):= \{ \pa{x_k}_{k\in I} \in \ell_2(I,\C):  \quad |x|_s:=\sum_{k\in I} \jap{k}^2 e^{2s\sqrt{\jap{k}}}|x_{k}|^2  <\infty \}\,,\quad \jap{(j,\s)} :=\max(|j|,1)\,
  \]
  and consider analytic vector fields and holomorphic functions of variables belonging to $\tg_s$. As it is habit, we shall introduce a quite natural norm on these sets, which endows them with a structure of filtred Lie-Poisson algebra, see section \ref{funzionale} for details. 
 \item We fix $\lambda\in\C^I$, with $\lambda_k\neq 0,\,\forall k\in I$, satisfying an appropriate arithmetic-Diophantine condition  together with an assumption on their asymptotic behavior, see Assumption \ref{assumption2}. This allows us to properly define the diagonal vector fields $\tD(\lambda)$ mentioned before. 
\item We consider the sets $\cM_\lambda, \Delta_\lambda$, in line with the finite dimensional case, we denote by $Q_i,P_i$'s the (countable) generators and assume that they have uniformly bounded size. This is again an hypothesis on $\lambda$, that allows us to define the manifold $\Sigma$, the non diagonal resonant vector fields and $\tM^*<\infty$.
  \end{itemize}
Now we consider a vector field as in \eqref{campo}, under the further assumption that $Z,P$ are analytic on $\tg_s$. Then, in Theorem \ref{main-thm}, we prove a normal form result as \eqref{normale} where $\phi$ is a holomorphic diffeomorphism on a ball in $\tg_s$. Note that under such weak hypotheses one cannot even guarantee local well posedness of the flow of $X$. See Remark \ref{rem-wp} for a more detailed discussion on this issue.
 
\section{Analytic vector fields, constants of motions and resonances }\label{funzionale}

%
%
 
\subsection{Functional setting}

 Let $B_r(\tg_s)$ be the open ball of radius $r$ in $\tg_s$, that is
$$
B_r(\tg_s):=\{x\in\tg_s\,|\,|x|_s<r \}\,.
$$

\noindent 
We now introduce analytic functions and analytic vector fields that ``preserve momentum", namely those that are invariant w.r.t. the following action
\begin{equation}
\label{azione invariante}
T_\fm:\; x_k\mapsto e^{\im \fm_k}x_k\,,\quad \quad \fm_{(j,\s)}:= \s j\,.
\end{equation}

\begin{defn}[Holomorphic functions]
\label{hol}
Given $r,s>0$, we  let { $\cH_{s, r}$} be the set  of normally analytic functions $f: B_r(\tg_s)\to \C$ defined as absolutely convergent power series with bounded majorant norm
$$
f(x)=\sum_{q\in \N^I_{\fin}} f_q x^q \quad \text{with } \quad |f|_{s,r} := \sup_{|x|_s < r} \sum_{q\in \N^I_{\fin}} |f_q| x^q < \infty,
 $$
that are invariant w.r.t \eqref{azione invariante} namely
\[
f_q = 0  \quad  \mbox{if}\quad \sum_{ h\in I} \fm_h q_{h}  =  \fm \cdot q \ne 0\,.
\]
\end{defn}
In line with analytic functions, in order to define analytic vector fields we need to introduce \textbf{monomial vector fields}, that is $$ x^q  \frac{\partial}{\partial x_k}, \quad k\in I,\quad  q \in \N^I_{\fin}\,.$$

\begin{defn}[Admissible vector fields]\label{cv} Given $r,s >0$ let $\cV_{r,s}:= \cV(\tg_s,\norma{\cdot}_{r,s})$ be the space of analytic vector fields on $\tg_s$ defined as absolutely convergent power series of the form
\[
V(x)= \sum_{k \in I}V^{(k)}(x)\base{x_k}:=\sum_{k\in I, q\in \N^I_{\fin}} V_q^{(k)} x^q \frac{\partial}{\partial x_k}
\]

such that 

\begin{enumerate}
\item $V(0) = 0$
	\item   $V$ is invariant w.r.t \eqref{azione invariante}, namely
\[
V_q^{(k)} = 0  \quad  \mbox{if}\quad \sum_{ h\in I} \fm_h q_{h}  - \fm_k =  \fm \cdot (q-\be_k) \ne 0.
\]
{We shall say that $V$ is {\it momentum preserving}.}
\item The \textbf{majorant norm} of ${V}$ is bounded
\[
\|V\|_{r,s}:= \frac{1}{r }\sup_{x\in B_r(\tg_s)} |\und{V}(x)|_s< \infty\,,\quad \quad \und{V}(x)= \sum_{k\in I, q\in \N^I_{\fin}} |V_q^{(k)}| x^q \frac{\partial}{\partial x_k}\,. \]
\end{enumerate}
\end{defn}

\medskip

The majorant norm endows both $\cH_{s, r}$ and $\cV_{s, r}$ with a Banach space structure. 
We define the {\it homogeneous degree} at zero of  functions by setting $$\mbox{deg}( x^q):= \|q\|_{\ell_1}=\sum_{k \in I} q_k$$ and of vector fields as
\[
\mbox{deg}( x^q \frac{\partial}{\partial x_k}):=\|q\|_{\ell_1}-1= \sum_{k \in I} q_k  -1,
\] 
we denote $\cH_{r,s}^\td$, resp. $\cV_{r,s}^\td$, the space of homogeneous functions, resp.  vector fields, of degree $\td$. Considering a vector field $V=V_1+V_2+\cdots$ with $V_i\in \cV_{r,s}^i$, we shall say that $V$ is of {\it order} $\td$ if $\td $ is the smallest integer such that $V_{\td}\neq 0$ and $V_i=0$, $i<\td$.  Similarly, we denote $\cH_{r,s}^{\ge \td}$ (resp. $\cH_{r,s}^{> \td}$) the space of vector fields of order $\ge \td$ (resp. $> \td$).\\
Note that $(1)$ ensures that the degree $\td\ge 0$ for vector fields in $\cV_{s,r}$.
\begin{lemma}[Inclusion of spaces]\label{monotone}
 $\cH_{s,r}$ and $\cV_{s,r}$ are scales of Banach spaces w.r.t. $s$, namely
 \[
 \cH_{s,r} \subseteq \cH_{s',r} \,,\quad  \cV_{s,r} \subseteq \cV_{s',r} \qquad s\ge s'\,,
 \]
 more precisely
 \[
 |\cdot|_{s,r} \le  |\cdot|_{s',r'}\,,\quad \quad \|\cdot\|_{s,r} \le  \|\cdot\|_{s',r'} \; \quad \forall  s\ge s'\,,r\le r'\,,
 \]
 The norms are also compatible with the degree namely  for all $f\in \cH_{s,r}^{\td}$, resp  $X\in  \cV_{s,r}^{\td}$
 \[
  |f|_{s,r} =  \pa{\frac{r}{r'}}^\td |f|_{s,r'}\,,\quad  \|X\|_{s,r} =  \pa{\frac{r}{r'}}^\td \|X\|_{s,r'}
 \]
\end{lemma}
\begin{proof}
See Appendix \ref{appendicite}
\end{proof}

\begin{defn}[Projections]
Given a subset $J\subseteq \N^I_\fin$ we define a projection $\Pi_J$ on functions as 
\begin{equation}\label{proj j}
\Pi_{J} \sum_{q\in\N^I_\fin} f_q x^q:=  \sum_{q\in J} f_q x^q\,,
\end{equation}
 equivalently  given a subset $J\subseteq \N^I_\fin\times I $ we define a projection $\Pi_J$ on vector fields
 \begin{equation}\label{proj j2}
 \Pi_J \sum_{k\in I} \sum_{q\in\N^I_\fin} X_q^{(k)}x^q\papa{k} :=  \sum_{(q,k)\in J} X_q^{(k)}x^q\papa{k}
 \end{equation}
\end{defn}
A special case is the projection on the degree.
\begin{defn}[Degree projections]
Given $\td\ge 0$ we define $\Pi^{(d)}:  \cH_{s,r}\to  \cH_{s,r}^{\td}$ as
\[
\Pi^{(\td)} \sum_{q\in\N^I_\fin} f_q x^q:=  \sum_{q\in\N^I_\fin:\|q\|_{\ell_1}=\td} f_q x^q
\]
analogously we define
$\Pi^{(d)}:  \cV_{s,r}\to  \cV_{s,r}^{\td}$ as
\[
\Pi^{(\td)}\sum_{k\in I} \sum_{q\in\N^I_\fin} X_q^{(k)}x^q\papa{k} := \sum_{k\in I}  \sum_{q\in\N^I_\fin:\|q\|_{\ell_1}=\td+1} X_q^{(k)}x^q\papa{k}
\]
\end{defn}

It is straightforward from the definition of the norms \ref{hol}-\ref{cv}, that the above projectors are idempotent continuous operators, with operator norm equal to $1$.
\smallskip
\subsection{The Lie derivative operator}
 In Proposition \ref{fan} we show that a regular vector field $X$ is locally well posed and gives rise to a flow $\Phi^t_X$ at least for small times. This allows us to define the Lie derivative operator $L_X$. \\ Given $X\in \cV_{s,r + \rho}$,  we define
$$ L_X : \cH_{s, r + \rho} \to \cH_{s, r} $$ 
 $$  f\mapsto L_X f := X[f] = \frac{d}{dt} \Phi^{t,*}_X f_{\vert_{t=0}}\,.$$

	 Accordingly, $L_X$ acts on vector fields through the adjoint action
	 $$L_X : \cV_{s, r + \rho} \to \cV_{s, r}$$
	 \[
Y\mapsto	L_X Y := \ad_X\pa{Y}= \frac{d}{dt} \Phi^{t,*}_X Y_{\vert_{t=0}}=\sq{X,Y} \,.
	 \]

More explicitly, the definitions above yield
\[
L_X f = \sum_{k \in I}V^{(k)}(x)\frac{\partial f}{\partial x_k}
\]
and
\[
L_X Y = [X,Y] = \sum_{j} \big(X[Y^{(j)} ]-Y[X^{(j)}]\big)\base{x_j}\,
\]
so that the invariance property in item $(2)$ of Definition \ref{cv} represents the fact that $V$ Lie commutes  with 
\[
 M = \im \sum_{k \in I} \fm_k x_k \frac{\partial}{\partial x_k}\,.
\]
\\
In this way, the families $(\cH_{r,s})_{r,s>0}$, $(\cV_{r,s})_{r,s>0}$ are a scale of graded Lie/Poisson Banach Algebras, as formalized in the following Proposition.

\begin{prop}\label{fan}
	For $0 <\rho\leq r$, $f\in \cH_{s,r+\rho}$, $X,Y\in \cV_{s,r+\rho}$  we have
	\begin{equation}\label{commXHK}
	|L_X f |_{s,r}\le \pa{1+\frac{r}{\rho}}
	\|X\|_{s, r}|f|_{s, r+\rho}\,,\qquad	\|L_X Y\|_{s, r}
		\le 
		4\pa{1+\frac{r}{\rho}}
		\|X\|_{s, r+\rho}
		\|Y\|_{s, r+\rho}\,.
	\end{equation}
\end{prop}
\begin{proof}
 The first bound follows directly from classic Cauchy estimates on analytic functions, while the second one is derived in essentially the same way as the analogous one in \cite[Lemma 2.15]{BBiP2} with $n=0$ , the only difference being the fact that here there are no action variables, which scale differently from the cartesian ones, this implying that the constant in the present paper is $4$ instead of $8$. The only properties on which the proof relies are the ones enjoyed by the Hilbert space of sequences $\tg_s$, that are the same as the space $E$ in \cite[Definition 2.5]{BBiP2}.
\end{proof}

From Proposition \ref{fan} it is therefore straightforward to deduce the following Lemma.
\begin{lemma}[Flow]\label{ham flow}
	Let $0<\rho< r $,  and $S\in \cV_{s,r+\rho}$ with 
	\begin{equation}\label{stima generatrice}
	\norma{S}_{s,r+\rho} \leq\delta:= \frac{\rho}{8 e\pa{r+\rho}}. 
	\end{equation} 
	Then the time $t$-flow 
	$\Phi^t_S: B_r(\tg_s)\to
	B_{r + \rho}(\tg_s)$  for $|t|\le 1$ is well defined, analytic, symplectic. Moreover 
	\begin{equation}
		\label{pollon}
		\sup_{u\in  B_r(\tg_s)} 	\norm{\Phi^1_S(u)-u}_{s}
		\le
		(r+\rho)  \norma{S}_{s, r+\rho}
		\leq
		\frac{\rho}{8 e}.
	\end{equation}
	For any $X\in \cV_{s,r+\rho}$
	we have that
	${\Phi^1_S}_* X= e^{\sq{S,\cdot}} X\in\cV_{s, r}$ and
	\begin{align}
		\label{tizio}
		\norma{\es X}_{s,r} & \le 2 \norma{X}_{s,r+\rho}\,,
		\\
		\label{caio}
		\norma{\pa{\es - \id}X}_{s,r}
		&\le  \delta^{-1}
		\norma{S}_{s,r+\rho}
		\norma{X}_{s,r+\rho}\,,
		\\
		\label{sempronio}
		\norma{\pa{\es - \id - \sq{S,\cdot}}X}_{s,r} &\le 
		\frac12 \delta^{-2}
		\norma{S}_{s, r+\rho}^2
		\norma{X}_{s, r+\rho}	\end{align}
	More generally for any $h\in\N$ and any sequence  $(c_k)_{k\in\N}$ with $| c_k|\leq 1/k!$, we have 
	\begin{equation}\label{brubeck}
		\norma{\sum_{k\geq h} c_k \ad^k_S\pa{X}}_{s, r } \le 
		2 \|X\|_{s, r+\rho} \big(\|S\|_{s, r+\rho}/2\delta\big)^h
		\,,
	\end{equation}
	where  $\ad_S\pa{\cdot}:= \sq{S,\cdot}$.
\end{lemma}

Let us now analyze  the adjoint action of a special class of vector fields of degree $0$. 

\subsection{Diagonal vector fields and commuting flows} 
We denote 
\[ 
\tD_k := x_k  \frac{\partial}{\partial x_k} \,,\quad \mbox{and for } \lambda\in \C^I\,,\quad \tD(\lambda)= \sum_k \lambda_k \tD_k.
 \]
 

\begin{remark}\label{rem-wp}
If $\lambda\in\ell^\infty(I,\C)$, then $\tD(\lambda)\in\cV_{s,r}$, for any choices of indexes. On the other hand, if this condition is not met the equations might not even be locally well posed. In any case, even if we can define a solution map, this last one is not $C^1$ in time with values in $\tg_s$.  To be explicit, the equation $\dot x_k =\lambda_k x_k$  is always defined on $\C^I$  and has solution
$x_k(t)= e^{\lambda_k t} x_k(0)$ but if $\sup_k |{\rm{ Re}}\lambda_k| =\infty$ then  $x(0)\in \mathtt g_s$ does not imply $x(t)\in \mathtt g_s$ even for short times. Naturally if all ${\rm Re}\lambda_k$ except at most a finite number have the same sign, then the solution is well defined for either positive or negative time.
\\
If on the other hand $\sup_k |{\rm Re}\lambda_k|= L<\infty$ but  $\sup_k |{\rm Im}\lambda_k|= \infty$ then $|x_k(t)|\le e^{L|t|}|x_k(0)|$, so that $x(t)\in \mathtt g_s$ for all times but $\dot x(t)$ might not belong to $\mathtt g_s$ even for short time.
\end{remark}
 
 Even though $\tD(\lambda)$ may not be locally well posed on $\tg_s$ we shall define its Lie derivative operator 
$L_{\tD(\lambda)}$, which we shall denote for compactness $L_\lambda$, as a linear operator from $\cH_{s,r}$ (resp. $\cV_{s,r}$) into the space of formal power series.  Direct computations show that $L_{\lambda}$ is well defined and acts diagonally on monomials (both functions and vector fields). More precisely
\begin{equation}
\label{azione diag funzione}
L_\lambda x^q = (\lambda\cdot q) \, x^q\,,\quad \quad L_\lambda x^q \frac{\partial}{\partial x_k} = \lambda\cdot (q-\be_k)  x^q \frac{\partial}{\partial x_k}
\end{equation}

\noindent Thus, a monomial vector field is  in $\ker(\ad_{\tD(\lambda)})$ if and only if  $\lambda\cdot (q-\be_k)=0$. \\
 
 \smallskip
 
 By linearity, the vector field $$V(x)=\sum_{k\in I, q\in \N^I_{\fin}} V_q^{k} x^q \frac{\partial}{\partial x_k} \quad\text{commutes with}\quad   \tD(\lambda) \quad \Leftrightarrow \quad
 V^{k}_{q}(\lambda\cdot q - \lambda_{k}) = 0$$ for any $k\in I $ and  $q\in\N^{I}_{\fin}$. \\
 Similarly, a function $$f(x)=\sum_{q\in \N^I_{\fin}} f_q x^q \quad \text{ is a first integral for } \tD(\lambda) \quad \Leftrightarrow \quad f_{q}(\lambda\cdot q) = 0$$ for any $q\in\N^I_\fin$.
 
 \smallskip

Let us decompose the space of analytic vector fields as 
\[
\cV_{s,r}= \cK_{s,r} \oplus  \cR_{s,r} \,,\quad  \cK_{s,r}:=\{ V\in \cV_{s,r}: \quad L_{\lambda}V=0\, \}\,.
\]
Of course the same decomposition holds degree by degree and we shall denote with an apex $\td$ the corresponding subspaces. \\
Similarly, denoting by $\cC_{s,r}\subset \cH_{s,r}$  the {ring} of first integrals, i.e. analytic functions which are invariant w.r.t the action of   $L_\lambda$, we decompose 
$$\cH_{s,r} = \cC_{s,r}\oplus\cC_{s,r}^\perp\,. $$
 \begin{defn}[Diagonal vector fields]
 	We denote by $\cV_{s,r}^{\rm diag}$ the set of vector fields $Y\in\cV_{s,r} $ such that 
 	$$
 	Y = \sum_{k\in I, p\in\N^I_\fin} Y^{(k)}_{p + e_k} x^p x_k \base{x_k}  \equiv \sum_{k\in I} \cY^{(k)}(x) x_k \base{x_k}\,,\quad \cY^{(k)}\in \cH_{s,r}
 	$$ 
 	so that $\cV_{s,r}$ can be decomposed in the direct sum of its diagonal part and the complementary which we denote by $\cV_{s,r}^{\rm out}$. 
 \end{defn}
 We note that the 
  action of $L_\lambda$ on $Y$ is given by 
 $$
 {L_\lambda} Y = \sum_{k\in I} (L_\lambda\cY^{(k)}(x)) x_k\base{x_k}\,,
 $$
 
moreover  a diagonal vector field $Y\in\cR_{s,r}$  if and only if 
 $$
 \cY^{(k)}(x) = \sum_{p\in\N^I_\fin} \cY^{k}_p x^p = \sum_{\substack{p\in\N^I_\fin \\ \lambda\cdot p \neq 0}} \cY^{k}_p x^p \,\quad \in\quad \cC_{s,r}^\perp\,.
 $$
 
\subsection{Structure of resonant monomials}

Let us now describe $\cK_{s,r}$ and $\cC_{s,r}$ in terms of restrictions on the indexes of their Taylor series, following \eqref{azione diag funzione}.\\ Some definitions are in order.\\
$\bullet$ A relation of the kind $q\cdot\lambda = \lambda_k$ for some $q$ and some $k$, is called \emph{resonance relation} and the corresponding monomial vector field is said to be resonant. \\
 \noindent
$\bullet$ It is natural to define the {\emph{ring of first integrals} of $\tD(\lambda)$} and \emph{module of resonance}  respectively
 
 \begin{equation}\label{dignitoso}
 	\begin{aligned}
 	{\cM}_\lambda&:= \{ Q\in \N^I_{\tf}:   Q\cdot \lambda =0, \; Q\cdot\fm = 0\}\\
 	{\Delta}_\lambda&:= \cup_{k\in I}{\Delta}_\lambda^{(k)}:=\cup_{k\in I}\{ P\in \Z^I_{\tf}:  P\cdot \lambda =0,\;  P\cdot \fm=0 \;\mbox{and} \; P+e_k\in \N^I_{\tf}    \}
 	\end{aligned}
 \end{equation}
\begin{lemma}\label{diago}
We have the following characterization of the kernel of $L_\lambda$ in terms of $\cM_\lambda$ and $\Delta_\lambda$:
\begin{equation}
\cC_{s,r} = \overline{\operatorname{Span}}(x^Q)_{Q\in\cM_\lambda}\,, \quad \quad \cK_{s,r} = \overline{\operatorname{Span}}(x^{P+e_k}\frac{\partial}{\partial x_k})_{P\in\Delta^{(k)}_\lambda,\,k\in I},
\end{equation}
where the closure is w.r.t. the norms $\norm{\cdot}_{s,r}$ and $\norma{\cdot}_{s,r}$ respectively.
\end{lemma}

The Kernel of $L_\lambda$ can be also decomposed in terms of diagonal vector fields, that is
 \begin{equation*}
 \cK_{s,r} = \cK_{s,r}^{diag} \oplus \cK_{s,r}^{out}
 \end{equation*}
 where 
 \begin{equation}\label{pappa}
\cK_{s,r}^{diag} = \overline{\mbox{Span}}_\C( x^Q x_k\frac{\partial}{\partial x_k})_{Q\in \cM_\lambda,k\in I} \quad \quad  \cK_{s,r}^{out} = \overline{\mbox{Span}}_\C( x^{P + e_k}\frac{\partial}{\partial x_k})_{\substack{k\in I,\\ P\in \Delta^{(k)}_\lambda\setminus \cM_\lambda}}
 \end{equation}

\begin{prop}
 $\cM_\lambda$ is generated by an at most  countable set. Namely there there exists $\cN\subseteq \N$ and a set
$\cG_\lambda:=\{ Q_i\}_{i\in \cN}\subset \cM_\lambda$, such that each element $Q\in \cM_\lambda$ is written in a \emph{unique} way as a finite sum of $Q_i$ as
$$Q = \sum_{i\in\cN} n_i Q_i\,.$$
\\
Similarly there exists $\bar\cN\subseteq \N$ and a set  $\bar{\cG}_\lambda = \set{P_j}_{j\in\bar{\cN}}\in \Delta_\lambda\setminus\cM_\lambda$ such that for each $P\in \Delta_\lambda\setminus\cM_\lambda$ there exist unique $P_j\in \bar{\cG}_\lambda$ and $Q\in \cM_\lambda$ such that 
$$
P=P_j+Q
$$
\end{prop}
\begin{proof}
Consider a monomial first integral $u$. As it is finitely supported, say involving coordinates $(j,\sigma)$, $|j|\leq n$, it also a first integral of the restricted system to $(j,\sigma)$, $|j|\leq n$~: $$
D_n(\lambda)[u]:=\sum_{k\in I,|k|\leq n}\lambda_k x_k\frac{\partial u}{\partial x_k}=0.$$
It is known that the ring of first integral of $D_n(\lambda)$ is generated by a finite number of monomials $M_n$ (see \cite{walcher}[proposition 1.6] or, in more general setting \cite{Stolo-ihes}[proposition 5.3.2]). As we have $M_n\subset M_{n+1}$, there is at most a countable numbers of generators.
\end{proof}

\subsection{Assumptions and Diophantine conditions}

From now on, we shall work under the following restrictions on $\lambda$:
\begin{ass}\label{ass bounded degree}
We shall assume that  $\lambda_k\ne 0 ,\quad \forall k\in I$ and that  the generators $Q_i,P_i$ have uniformly bounded degree 
 \[
 \sup_{i\in\N} \|Q_i\|_{\ell_1} \le \tM\,, \qquad 	\sup_{k\in I}\sup_{P_j\in \Delta_{\lambda}^{(k)}}\|P_j+e_k\|_{\ell_1}\leq \tM_1.
 \]
\end{ass}

 Note that this implies that, for all non-zero $P\in \Delta_\lambda$  one has $\|P\|_{\ell_1}\ge 2$.\\ 
 
 \begin{ass}\label{assumption2}
 	We shall assume that the frequency  vector $\lambda$ is superlinear,
 	namely there exist $\{e^{\im \varphi_{k}}\}_{k\in I}$ such that $\lambda$ 
 	belongs to the square
 	\begin{equation}\label{qalpa}
 		\mathtt Q:=\{\lambda\in \C^\Z:\quad |\lambda_{k} - \lambda^{(0)}_k |\le \frac12\}\,,\quad  \lambda^{(0)}_k:= \jap{k}^\al e^{\im \varphi_{k}}\,,\quad \alpha>1.
 	\end{equation}
 Furthermore we require that  there exists $C>0$ such that for all $(j,\s)\in I$ such that 
 		$\lambda_{(j,\s)}\ne \lambda_{(-j,-\s)}$ one has
 
 	\begin{equation}\label{assunzione bound}
 			| e^{\im \varphi_{(j,\s)}} - e^{\im \varphi_{(-j,-\s)}}|\ge C\,.
 	\end{equation}

 \end{ass}
The assumption above is crucial in solving the Homological equation (see Proposition \ref{adjoint action}). Note however that the bound \eqref{assunzione bound} can be weakened as $\geq \frac{C}{\jap{j}^\beta},$ with $\beta$ small. This just makes the proof slightly more technical in Case 3C in the proof of Proposition \ref{adjoint action}.

 \begin{defn}\label{def-dioph}
 	We shall say that $\lambda$ is $(\gamma,\tau)$-Diophantine modulo $\Delta_\lambda$ if 
 	$$ 
 	|\lambda\cdot p | \geq \gamma \prod_{i\in I} 	\frac{1}{(1 + p_i^2 \jap{i}^2)^\tau}
 	$$
 	for any $p\in\Z^I_\tf \setminus \Delta_\lambda$ such that $p\cdot\fm=0$ {and there exists $k\in I$ such that $p + e_k \in \N^I_\fin$}.
 \end{defn} 
 
 It is well known  -see for instance \cite{Bourgain:2005} - that  $(\gamma,\tau)$-Diophantine  vectors have positive measure  in $\mathtt Q$ for $\tau>\frac12$ and $\gamma$ small enough.

\noindent

Given $Q_i\in\cG_\lambda\,,i\in\cN$ let us define the corresponding {\it resonant}  analytic sets

 \begin{equation}\label{res manif}
 	\Sigma_i := \set{x\in B_r(\tg_s)\, : \, x^{Q_i} = 0 } \quad \quad \Sigma := \bigcap_{i\in\cN} \Sigma_i\,
 \end{equation}
As $\cG_{\lambda}$ is at most countable, we can order the monomials $x^{Q_i}$'s and define the map $f: x\mapsto (x^{Q_i})_{i\in \cN}$ on $B_r(\tg_s)$. Its image lies in the Banach space $E:=\{(x^{Q_i})_{i\in \cN}, x\in \tg_s\}$ (we recall that there is no algebraic relations among the resonant monomials) and $f$ is analytic. Hence, according to \cite{ramis-book}[proposition II.1.1.1 (iii)], $\Sigma=f^{-1}(\{0\})$ is an analytic subset. We refer again to \cite{ramis-book} for general facts on analytic sets in Banach spaces.\\
%
%
%
%
\subsection{Vector fields tangent to $\Sigma$}
Let us now characterise those vector fields that vanish on $\Sigma$. As usual, we do this in terms of monomial vector fields. To this purpose, let us introduce the following sets
\[
\cJ^{(1)}_\lambda:=  \{ q\in \N^I_{\tf}:   \quad \exists i\in \cN\;\mbox{such that} \quad q-Q_i\in \N^I_{\tf}\},
\]
\[
\cJ^{(2)}_\lambda:=  \{ q\in \N^I_{\tf}:   \quad \exists i,j\in \cN\;\mbox{such that} \quad q-Q_i-Q_j\in \N^I_{\tf}\},
\]
\[
\cJ^{(0)}_\lambda := \N^I_{\tf}\setminus \cJ^{(1)}
\]
and decompose
\begin{equation}
	\label{decomp}
	\cV_{s,r}=  \cI^{(0)}_{s,r}\oplus  \cI^{(1)}_{s,r}\oplus \cI^{(2)}_{s,r}
\end{equation}
where
\begin{eqnarray}
\label{ideali0}
	 \cI^{(0)}_{s,r}:=\{X\in \cV_{s,r}: X =\sum_{k\in I,q\in  \cJ^{(0)}} X_q^k x^q \frac{\partial}{\partial x_k}\}\\
 \label{ideali1}
	 \cI^{(1)}_{s,r}:=\{X\in \cV_{s,r}: X =\sum_{k\in I,q\in  \cJ^{(1)}\setminus \cJ^{(2)}} X_q^k x^q \frac{\partial}{\partial x_k}\}
	 \\
	 \label{ideali2}
	\cI^{(2)}_{s,r}:=\{X\in \cV_{s,r}: X =\sum_{k\in I,q\in   \cJ^{(2)}} X_q^k x^q \frac{\partial}{\partial x_k}\}\,.
\end{eqnarray}

\begin{remark}\label{rem momentum}
Recall that by our definition, vector fields and functions are momentum preserving. Thus in the subsets above, $X^k_q (\fm\cdot q - e_k) = 0$. 
\end{remark}
\begin{lemma}\label{emme*}
There exists a degree $\tM^*<\infty$ such that one has
	\begin{equation}\label{condiz range}
		\cI^{(0)} \cap \cK^{\ge \tM^*}=\{0\} \,,\quad  \cI^{(1)} \cap \cK^{\ge \tM^*} =\{0\}.
	\end{equation}
	In other words, resonant terms of high enough degree are divisible by monomials $x^{q_1+q_2}$, $q_i\in\cM_{\lambda}$.\\
\end{lemma}
\begin{proof}
Let $x^q\frac{\partial}{\partial x_k}\in\cK$, of degree $\geq\tM^* = 2\tM + \tM_1$  with $q-e_k =:P \in\Delta_\lambda$,  hence equivalently  $x^q\frac{\partial}{\partial x_k} = x^{P + e_k}\frac{\partial}{\partial x_k}$. We have $ \|q\|_{\ell_1}-1\geq \tM^*$. 
	If $P\in\cM_\lambda$,  then $P = \sum_{i} n_i Q_i$ and $ \|P\|_{\ell_1}+1= \|q\|_{\ell_1}$. Since $ \|P\|_{\ell_1} \geq \tM^*>2M$,  then  necessarily $\sum_i n_i \ge 2$, so $q = P + e_k\in\cJ^{(2)}_\lambda$.\\
	If $P\in\Delta_\lambda\setminus\cM_\lambda$,  then $P\in\Delta_\lambda^{(k)}$ for a unique $k$, so that $P_k=-1$. 
	By our assumption,  there exist $P_j\in\bar\cG_\lambda$ and $Q\in\cM_\lambda$ such that $P = P_j + Q = P_j + \sum_{i} n_i Q_i$.  On the other hand, we have
	\begin{align*}
		\tM^* : = 2\tM + \tM_1 \le  \|q\|_{\ell_1}-1= \|P +e_k\|_{\ell_1}-1 &\le \sup_{P_j\in\Delta_{\lambda}^{(k)}} \|P + e_k\|_{\ell_1} + \sup_i  \|Q_i\|_{\ell_1}\sum_i n_i -1\\
		&\le \tM_1 + \tM \sum_i n_i -1\ \,,
	\end{align*}
	the sums being finite.
	Hence, $1\leq \tM(\sum_i n_i -2)$ implies $\sum_i n_i >2$ and the conclusion follows.	
\end{proof}

\section{Main Result and examples}

\begin{Thm}\label{main-thm}
	Let $\lambda\in \mathtt Q$ be $(\gamma,1)$-Diophantine modulo $\Delta_\lambda$.	Let $W \in \cV_{\mathtt s,\mathtt R}^{\ge 0}$ be a vector field of the following form $$W = \tD(\lambda) + Z + X,\quad \quad X\in \cV_{\mathtt s,\mathtt R}^{\ge \tM_*},  \quad Z \in \cK^{diag}_{s,\mathtt{R}}\cap \cV^{(1\le \td \le \tM_* - 1)}_{s,r} $$ with $\mathtt s,\mathtt R>0$.
	Then,  for any $s' > \mathtt{s}$ there exists $r' < \mathtt{R}/2$ and an analytic change of variables $\phi : B_{r'}(\tg_{s'}) \to B_{2r'}(\tg_{s'})$, isotopic to the identity $\phi(x) = x + \psi(x) $ such that 
	\begin{equation}\label{coniugato}
		 \phi_* W = \tD(\lambda) + Z + Y\,,\quad Y\in \cI^{(2)}\, . 
		 \end{equation}
\end{Thm}
As a consequence, in the new coordinate system, the vector field $\phi_*W$ is not only tangent to $\Sigma$, the common zero set of the $X^Q$'s, $Q\in \cI^{(1)}$ , but also its restriction to it is linear and equal to $\tD_{\lambda|\Sigma}$, that is its flow is linear, with characteristic exponents $\lambda_j$. 
\begin{remark}
		As $\Sigma$ is defined by the vanishing of monomials of bounded degree, it is at most a countable intersection of union of coordinates hyperplanes. It is a union of "irreductible" components $\Sigma=\cup_i \Sigma_i$ passing through the origin.
		As  noted in Remark \ref{rem-wp} the vector field $\tD(\lambda)$ and $W$ might not define a well-posed flow even locally. Nevertheless, each irreducible component $\Sigma_i$ of $\Sigma$ can be decomposed as $\Sigma_i^s\cup\Sigma_i^u\cup \Sigma_i^c$, where the eigenvalues of $\tD(\lambda)$ restricted to  $\Sigma_i^s$ (resp. $\Sigma_i^u$, resp. $\Sigma_i^c$) have negative (resp. positive, resp. zero) real part. The restriction of $\tD(\lambda)$ to each of these sub-components give rise to a system whose dynamics is well defined for positive times on $\Sigma_i^s$ (resp. negative time for $\Sigma_i^u$ and for real time on $\Sigma_i^c$). Hence, by pulling-back each of these by the (same) analytic diffeomorphism, the phase space contains germs of at most countable analytic submanifolds passing through the origin, invariant by the dynamical system, the restrictions to which are simutaneously linearizable and whose flow is well defined either for positive or negative or for all real time.
\end{remark}
Let us now consider a  {\it Momentum preserving} (see Defintion \ref{cv}) vector field of the form $W^{(0)}= \tD(\lambda)+ X^{(0)}$, where $X^{(0)}$ has degree $\ge 1$.
 As explained in the Introduction, the notion of formal normal form with respect to D($\lambda$) is well defined and can be achieved by formal change of variables tangent to identity (see \cite{stolo-procesi}[Section 2] in the Hamiltonian setting). Moreover there exists an analytic change of coordinates that ensures that $W^{(0)}$ is conjugated to the form
\begin{equation}\label{dopoNF}
	  W = \tD(\lambda) + W^{Ker}_{\le \tM^*-1} + W_{\ge \tM^*}.
\end{equation}

If $W^{Ker}_{\le \tM^*-1}\in\cK^{diag}$ then we may apply our main Theorem thus obtaining a linearization result. Of course, if the subsets $\cM_\lambda$ and $\Delta_\lambda$ coincide, then the hypothesis is automatically met.
\subsection{Examples and Applications}
Before dealing with {infinite dimensional applications}, let us consider an example in dimension $6$ to illustrate some of the hypotheses. 
Eventhough our result is taylored for infinite dimension, it can be straightforwardly reformulated in the finite dimensional frame. In this case one does not need the momentum conservation and one can set $\fm=0, \mathtt{s} = s' = 0$. Note that in finite dimension the usual Diophantine condition is equivalent to Definition \ref{def-dioph} and could be used equivalently.

\begin{exam}
	As an example, let $\zeta_1,\zeta_2\ne 0$ be uncommensurable irrational numbers, that is $\zeta_1/\zeta_2\notin\Q$
	 and set $\lambda=(2,1,\zeta_1,-\zeta_1,\zeta_2,-\zeta_2)$. 
	
	 Let us consider a dynamical system in dimension 6 given by a nonlinear perturbation of the linear vector field $$\tD(\lambda):=2x_1\partial_{x_1}+x_2\partial_{x_2}+\zeta_1\left(x_3\partial_{x_3}-x_4\partial_{x_4}\right)+\zeta_2\left(x_5\partial_{x_5}-x_6\partial_{x_6}\right).$$

	 One directly verifies that $\cM_\lambda$ is generated by $Q_1= \be_3 +\be_4$ and $Q_2= \be_5 +\be_6$ while $\Delta_\lambda\setminus\cM_\lambda$ is obtained from  $\cM_\lambda$ by tralsation with $P_1= 2\be_2 -\be_1$ so that  $\tM_*=4$.
	 Hence, the constants of motions are generated by the monomials $x_3x_4, x_5x_6$, $\Sigma =\{x_3x_4=0\}\cap \{x_5x_6=0\} $ and formal resonant vector fields are of the form
	 
	 \begin{align*}\tD(\lambda)+cx_2^2\partial_{x_1}+ f_1(x_3x_4,x_5x_6)x_2^2\partial_{x_1}+f_2(x_3x_4,x_5x_6)x_2\partial_{x_2}+f_3(x_3x_4,x_5x_6)x_3\partial_{x_3}\\
		+ f_4(x_3x_4,x_5x_6)x_4\partial_{x_4}+ f_5(x_3x_4,x_5x_6)x_5\partial_{x_5}+ f_6(x_3x_4,x_5x_6)x_6\partial_{x_6},
	\end{align*} where $c$ is a constant and the $f_i$'s are  formal power series of two variables, vanishing at the origin.
	As mentioned in the introduction,  if $c=0$, then the set $\Sigma$ is invariant by the vector field above and its restriction reduces to the linear vector field $2x_1\partial_{x_1}+x_2\partial_{x_2}+{(-1)^{i'}\zeta_1}x_i\partial_{x_i}+{(-1)^{j'}\zeta_2}x_j\partial_{x_j}$ on $\{x_{i'}=0\}\cap \{x_{j'}=0\}$, $i,i'\in\{3,4\}, j,j'\in\{5,6\}$, $i\neq i'$,$j\neq j'$. 
	
	Let us now consider the analytic vector field  $W^{(0)} = \tD(\lambda)+ X^{(0)}$ with $X^{(0)}$ of degree at least $1$.
	Let us show that for many choices of $\zeta_1,\zeta_2$ the vector $\lambda$ satisfies the Diophantine condition \ref{def-dioph}.
	To this purpose let $\omega=(1,\zeta_1,\zeta_2)$. It is well known that, for $\gamma$ small enough and $\tau>3$,  many choices of $\omega$ satisfies the usual Diophantine condition
	$$ 
	|\omega\cdot \ell| \geq \frac{\gamma}{|\ell|^\tau} \quad\quad \forall \ell\in\Z^3\setminus \{0\}.
	$$
	
	Let us now consider the subset $\Z^6_\star$ of those  $p\in\Z^6\setminus\Delta_\lambda$ such that there exists $k\in \{1,\ldots,6\}$ for which $p + e_k \in\N^6$. By definition one has 
	$$
	|\lambda\cdot p| = |2p_1 + p_2 + \zeta_1(p_3 - p_4) + \zeta_2(p_5 - p_6)| \,.$$
	We note that $p\in\Z^6_\star$  implies that 
	$\ell := (2p_1 + p_2,p_3 - p_4,p_5 - p_6)\neq 0$. Noting that $|\ell|\le 2|p|$ we have
	\[
	{|\omega\cdot \ell|=}|\lambda\cdot p| \ge \frac{\gamma}{2^\tau |p|^\tau} \ge \mbox{const}\, \gamma \prod_{i=1}^6 (1+i^2 p_i^2)^{- 2\tau}\,,
	\] thus verifying the Diophantine condition \ref{def-dioph}.
	  
	  \smallskip
Performing three steps of BNF on  $W^{(0)}$  we push it forward to  
	\[
W :=	\phi^{BNF}_* W^{(0)} = \tD(\lambda) + (c_1 +c_2x_3 x_4 + c_3 x_5x_6) x_2^2 \partial_{x_1} + W^{Ker,diag}_{\leq 3} + W_{\geq 4}, 
	\]
	where $\phi^{BNF}$ is a close to identity analytic change of variables defined in an appropriate ball. 
	
	If $c_1 = c_2 = c_3 = 0$, then our result applies and $W$ is conjugated to \eqref{coniugato} through $\phi$. As a consequence the manifold $(\phi\circ\phi^{BNF})^{-1}\Sigma$ is invariant by the flow of $W^{(0)}$ in a ball close to $0$, and carries the linear dynamics $\tD(\lambda)$.
\end{exam}
{\begin{exam}\label{ex: 2}
		Consider the following PDE system on the circle $\vartheta\in \T:=\R/2\pi\Z$: 
		\begin{equation}\label{Zac}
			\begin{cases}
				&\im z_t = z_{\vartheta\vartheta} - V\star z + (z w)^p z\\
				&-\im w_t = w_{\vartheta\vartheta} - V{\,\bar\star\,} w + (z w)^p w
			\end{cases}
		\end{equation}
	with $p\in \N$, $V= \sum_j V_j e^{\im j \vartheta}$ with $(V_j)_{j\in \Z}\in \ell_{\infty}(\Z,\R)$ and 
	\[
	(V\star z )(\vartheta) =  \sum_{j\in \Z} V_j z_j e^{\im j \vartheta}\,,\quad  (V{\,\bar\star\,} w )(\vartheta) =  \sum_{j\in \Z} V_{-j} w_j e^{\im j \vartheta}\,.
	\]
	
	Note that on the invariant subspace $w=\bar z$, the system \ref{Zac} coincides with the NLS equation of degree $2p+1$.
	
	Passing to the Fourier basis we obtain the system of equations
		\begin{equation}\label{Zac2}
			\begin{cases}
						\dot z_j & = \im (j^2+V_{j})z_j+ \im(z^{p+1}w^p)_j\\
					\dot w_j & =  -\im (j^2+V_{-j})w_j -\im (z^pw^{p+1})_j.
			\end{cases}
				\end{equation}
				where $(f g)_j:= \sum_{j_1\in \Z} f_{j_1}g_{j-j_1}$.
		In order to fit our notation we set $I=\Z\times \{\pm\}$, $x=(x_k)_{k\in I}$ with
		\[
		x_{j,+}= z_j \,,\quad x_{j,-} =w_{-j}
		\]		moreover we define $\lambda_{j,\s}= \im \s  (j^2+V_j)$. With this notation the PDE is  rewritten as the momentum preserving dynamical system with vector field $W^{(0)}= \tD(\lambda) +X$ where
	\[
		X =\im \sum_{(j,\s)\in I} \s ( \sum_{\sum_{i=1}^{p+1} j_i - \sum_{i=1}^{p} h_i =j} \prod_{i=1}^{p+1} x_{j_i,\s} \prod_{i=1}^p x_{h_i,-\s} )\frac{\partial}{\partial x_{j,\s}}\,.
	\]
	We claim that $X\in \cV_{\mathtt s,\mathtt R}$,  see Definition \ref{cv}, for all $\mathtt s\ge 0$ and $\mathtt R>0$. Let us start by showing that $X$ 
	satisfies   the momentum conservation condition. Let us write
	$q=(q_k)_{k\in I}$ as $q=(q_+,q_-)$ with $q_{\s}=(q_{j,\s})_{j\in \Z}$ then,  using the notations \eqref{notation0} we have that 
	\[
	X^{(j,\s)}_q= \begin{cases}
		\binom{p+1}{q_\s}  \binom{p}{q_{-\s}}  & \mbox{if}\;  |q_{\s}|=|q_{-\s}|+1=p+1 \; \mbox{and} \; \sum_{h\in \Z} h (q_{h,\s}- q_{h,-\s} )=j\\
		0 & \mbox{otherwise}\,.
	\end{cases} 
	\]
where  the condition $\sum_{h\in \Z} h (q_{h,\s}- q_{h,-\s} )=j$  is just $ \fm \cdot (q- e_{\s,j})=0$. It remains to show that $X$ is a bounded map on any ball $B_{\mathtt R}(\tg_{\mathtt s})$. Let us introduce some notation: given $f,g\in \tg_{\mathtt s}$ set 
	$$
(	f\times g)_{j,\s}:=  \sum_{j_1\in \Z} f_{j_1,\s}g_{j-j_1,\s} \,,\quad 	(f\bar\times g)_{j,\s}:=  \s \sum_{j_1\in \Z} f_{j_1,\s}g_{j_1-j,-\s} 
	$$
	so that 
	$$ X= \im (\underbrace{x\times \dots\times x}_{p+1 \mbox{ times}} )\bar\times (\underbrace{x\times \dots\times x}_{p \mbox{ times}} )\,
	$$
	{meaning, that
	$$
	X_{j,\sigma}= \im \left\{(\underbrace{x\times \dots\times x}_{p+1 \mbox{ times}} )\bar\times (\underbrace{x\times \dots\times x}_{p \mbox{ times}} )\right\}_{j,\sigma},\quad X=\sum_{(j,\sigma)\in I}X_{(j,\sigma)}\frac{\partial}{\partial x_{j,\sigma}}.
	$$}
	Then $X$ is bounded as a map $B_{\mathtt R}(\tg_{\mathtt s})\to \tg_{\mathtt s}$ because $\times$ and $\bar\times$ are continuous bilinear maps with values in $\tg_{\mathtt s}$. For a proof see for instance \cite[Lemma 5.5]{BMP:2019}.
\\
By construction $\lambda_{j,+}=-\lambda_{j,-}$ moreover, setting $\omega_j =  \lambda_{j,+}= j^2+V_j$, Bourgain proved that for a positive measure set of $V\in B_{1/2}(\ell_{\infty})$  $\omega$ is $(\gamma,\tau)$-diophantine provided that $\g$ is small and $\tau> 1$, namely one has
\[
|\omega\cdot \ell|> \g \prod_{j\in \Z} \frac{1}{(1+\ell_j^2 \jap{j}^2)^\tau} \,,\forall \ell \in \Z^\Z_{\tf}\setminus\{0\}\,.
\]
Thus setting
\[
\Delta_\lambda :=\{p\in \Z^I_{\tf}: \quad p_+= p_{-}\,,\;\mbox{and there exists }\;  k\in I:\; p+e_k\in \N^I_\tf\}
\]
we have that for all $p\in \Z^I_{\tf}\setminus \Delta_{\lambda}$ such that $\exists k\in I:\; p+e_k\in \N^I_\tf$ one has $p_+\ne p_-$ and
\[
|\lambda \cdot p| =|\omega \cdot (p_+-p_-)| > \g \prod_{j\in \Z} \frac{1}{(1+((p_{j,+}-p_{j,-})^2 \jap{j}^2)^\tau} > \g \prod_{(j,\s)\in I} \frac{1}{(1+p_{j,\s}^2 \jap{j}^2)^\tau}\,,
\] 
which implies that $\lambda$ is  $(\gamma,\tau)$-diophantine modulo $\Delta_{\lambda}$.

\smallskip
Moreover if $q\in \N^I_\tf$, $k=(j,\s)$ and $q-\be_k\in \Delta_{\lambda}$ then $ (q-\be_k)_{(j,\s)}= (q-\be_k)_{(j,-\s)} = q_{(j,-\s)}\in \N$ and one must have $q_k\ne 0$. This means that there are no non-diagonal resonant vector fields and  $\Delta_\lambda\equiv \cM_\lambda$. The generators of $\cM_\lambda$  are  indexed by $\cN=\Z$ 
$$
Q_j = e_{j,+}+ e_{j,-}  \quad \stackrel{\mbox{constants of motion}}{\Longrightarrow} \quad h_j= x_{j,+}x_{j,-}
$$
 so that $\tM= 2$ and Assumption \ref{ass bounded degree} is satisfied. By construction Assumption \ref{assumption2} is satisfied with $\alpha=2$ and $\varphi_{(j,\s)}= \s \pi/2$ so that  $C=2$.
 Now, following Lemma \ref{emme*}, we fix $\tM_*= 4$, $\mathtt s\ge 0$ and $\mathtt R>0$ sufficiently small. If $p\ge 2$, then $\tD(\lambda)+X$ satisfies all the hypotheses of our main Theorem (with $Z=0$). Otherwise, if $p=1$, we  perform  1 step of Birkhoff Normal Form, following for instance {\cite{BMP:2019}} essentially verbatim  since this is in fact a complex Hamiltonian PDE system. We obtain a vector field of the form \ref{dopoNF} to which we apply our main result. 
 \\
We have thus proved that there exists a close to identity change of variables $\phi$, defined in a neighborhood $B_{r'}(\tg_{s'})$ of the origin, such that setting
\begin{equation}\label{ex:sigma}
\Sigma:= \{x\in \tg_{s'}: \quad x_{j,+}x_{j,-}=0 \,,\;\forall j\}
\end{equation}
one has that $\phi_* W$ is tangent to $\Sigma$ and its flow, restricted to $\Sigma$ is the linear flow of $\tD(\lambda)$. As a consequence the following holds. 
Consider  any partition of
$\Z= S_+\cup S_-$ into two disjoint sets and  any $\xi\in B_{r'}(\tg_{s'})$ such that
    $\xi_{j,+}=0$   if $j\in S_-$ and  $\xi_{j,-}=0$   if $j\in S_+$.
    Let 
    \begin{equation}\label{ellitico}
  x_{\rm lin}(\xi,t):= (\xi_{j,\s} e^{\im \sigma(j^2+V_j) t})_{(j,\s)\in I}
  \end{equation}
    then for all $\xi\in B_{r'}(\tg_{s'})$ one has that  $\phi^{-1}(x_{\rm lin}(\xi,t))$ is an almost periodic solution of \ref{Zac2} with frequency $\omega=(j^2+V_j)_{j\in \Z}$. 
    \end{exam} }
It is worthwile to notice that when we restrict to the real subspace $z=\bar w$, where Equation \ref{Zac} is the usual NLS equation, our result trivializes, since  the intersection between $\Sigma$ and the real subspace is just $z=w=0$. In order to have a non-trivial example for real Hamiltonian systems we have to consider the neighborhood of an hyperbolic fixed point.
\begin{exam}\label{ex:hyperbolic}
	Consider a toy-model system defined on $\tg_{\mathtt s}(I,\R)$ with $I= \Z\times\{\pm\}$, sat $x= (x_{j,+},x_{j,-})_{j\in \Z}$. Consider the real Darboux symplectic form $\Omega= \sum_j d x_{j,+}\wedge d x_{j,-}$ and  the Hamiltonian
	\[
H = 	\sum_{j\in \Z} (j^2+V_j)x_{j,+}x_{j,-} + F(x)
	\]
	with $V\in \ell_\infty(\Z,\R)$ as in example \ref{ex: 2}  while $F(x)$ is a real-analytic function  on a ball $B_\mathtt R (\tg_{\mathtt s}(I,\R))$ with a zero of order at least three in $x=0$ and satisfying the momentum invariance
	$	F(T_\fm x) = F(x)$, see \eqref{azione invariante}. 
	Under the same non-resonant hypotesis on the frequencies $j^2 + V_j$ as in the previous example, we can proceed the same way and prove that there exists a close to identity change of variables defined in an appropriate neighborhood $B_{r'}(\tg_{s'})$, so that in these variables $\Sigma$ defined in \eqref{ex:sigma} is invariant. It must be noted that, for general initial data, the Hamiltonian flow of $H$ is not even locally well posed. Our result proves the existence of stable/unstable manifolds on which the dynamics is well posed either for positive or negative times. More in general, in terms of flows one can reason as follows.\\
	Consider  any partition of
	$\Z= S_+\cup S_-$ into two disjoint sets such that one of them, let's say $S_+$ is finite. Consider  any $\xi\in B_{r'}(\tg_{s'})$ such that
	$\xi_{j,+}=0$   if $j\in S_-$ and  $\xi_{j,-}=0$   if $j\in S_+$.
	Let \begin{equation}\label{iperbolico}
		x_{\rm lin}(\xi,t):= (\xi_{j,\s} e^{\sigma(j^2+V_j) t})_{(j,\s)\in I}
	\end{equation}
	then for all $\xi\in B_{r'/2}(\tg_{s'})$ one has that 
	$\phi^{-1}(x_{\rm lin}(\xi,t))$ is a solution belonging to $B_{r'}(\tg_{s'})$ at least for small positive times. Of course if $S_+$ is empty, then the above holds for all positive times.
\end{exam}

\begin{exam}\label{ex:mix}
	One can mesh examples \ref{ex: 2}-\ref{ex:hyperbolic}. Let us consider a partition of $\Z = S_{\text{ell}} \cup S_{\text{hyp}}$ and the Hamiltonian 
	$$H = 	\sum_{j\in \Z} (j^2+V_j)x_{j,+}x_{j,-} + F(x) $$
	but now with the symplectic form 
	$$\Omega= \im\sum_{j\in S_{\text{ell}}} d x_{j,+}\wedge d x_{j,-} + \sum_{j\in S_{\text{hyp}}} d x_{j,+}\wedge d x_{j,-}\,. $$ 
	Note that if  $S_{hyp}$ (resp. $S_{ell}$) is empty, we fall in example \ref{ex: 2} (resp. example \ref{ex:hyperbolic}). Of course, if hyperbolicity and ellipticity coexist, the solutions starting on $\Sigma$ and supported on $S_{\text{ell}}$ behave like in \eqref{ellitico} and are almost periodic. Otherwise, solutions starting on $\Sigma$ with support intersecting $S_{\text{hyp}}$ are well defined for small positive (resp. negative) times provided that all but a finite number of hyperbolic eigenvalues $j^2 + V_j$, with $j$ in such support, have the same sign. 
	
	Note however that in the case of a real system, one has $x_{j,+} = \bar{x}_{j,-}$ for all $j\in S_{\text{ell}}$ the manifold $\Sigma$ in the elliptic directions reduces to a point and the only nontrivial dynamics that survive is the hyperbolic one.
	\end{exam}

\section{Proof of the main result}

Let $\lambda\in \C^I$   be $(\g,\tau)$-Diophantine modulo $\Delta_\lambda$ and  satisfy Assumptions \ref{ass bounded degree}, \ref{assumption2}; for simplicity take $\al=2$. In dealing with a general $\alpha>1$ the only difference is the bound \eqref{stima homo 1}, where the exponent $6$ becomes more complicated. The proof would be essentially identical.

\subsection{Homological equation and Technical Lemmata}
In what follows, we omit the dependence on $r,s$ if the context permits.  Our goal is to prove the following

\begin{prop}[Straightening the dynamics]\label{homo prop}
	For any  $Z\in \cK^{\diag}\cap \cV_{s,r}^{\ge 1}$ ,   for any  $Y_i\in \cI^{(i)}\cap \cV^{\ge \tM_*}_{s,r} \,, i = 0, 1$  the equation
	\begin{equation}
		\label{homo estesa}
		\Pi^{\cI^{(i)}} [F , {\tD(\lambda)} + Z] = Y_i
	\end{equation}
	admits a unique solution $F_i \in \cI^{(i)}\cap \cV^{\ge \tM_*}_{s+\sigma,r - \rho}\, \forall \quad 0 < \rho < r\, , \s > 0$ satisfying the bound
	\begin{equation}
		\label{stima homo est}
		\|F_i\|_{s+\s,r - \rho} \lesssim \frac{r}{\rho} \g^{-1}{e^{2^7\tc/\s^6}} (1 + \g^{-1}\|Z\|_{s,r}) \|Y_i\|_{s,r}\,.
	\end{equation}
\end{prop}

This result follows from the definition of $\cI^{(i)}$ and from the statement below, regarding the invertibility of $L_\lambda$, which is proved in the Appendix.

 \begin{prop}[Homological equation] \label{adjoint action}
	Let $\lambda$ be as above.  For any $Y\in \cR_{s,r}$ the equation $L_\lambda X = Y$ admits a unique solution $X = L_\lambda^{-1} Y \in\cR_{s+\delta,r}$ for all $\delta > 0$ satisfying the bound
	\begin{equation}
		\label{stima homo 1}
		\|X\|_{s+\delta, r} \lesssim e^{\tc/\delta^6} \g^{-1} \|Y\|_{s,r}\,,
	\end{equation}
	for some positive $\tc$. 
\end{prop}
\begin{proof}
	In the Appendix.
\end{proof}

Let us  now make some  remarks on the strucutre of the spaces $\cI^{(i)}$.
\begin{remark}\label{rems}
	\begin{enumerate}
		\item  $\cI^{(1)}\oplus \cI^{(2)}$ is the subspace of analytic vector fields that vanish on $\Sigma$.\\
\item The subspace  $\cI^{(0)}_{s,r}$ does not contain diagonal  resonant vector fields of degree $\geq 1$, that is $\cK^{\rm diag,\ge 1}_{s,r}\cap \cI^{(0)}_{s,r} = \emptyset$.\\
In fact, any monomial in $\cK^{\rm diag,\ge 1}$ is of the form $x^Px_k\partial_{x_k}$ with $0 \ne P \in\cM_\lambda$, so that $P = \sum_i n_i Q_i$ with at least one $n_i\neq 0$.
\item
Setting
$
	\tm:= \min(\|P\|_{\ell_1})_{P\in \Delta_\lambda}\geq 2\,
$, one has the  following inclusion
\[
 \cK_{s,r} \cap \cV_{s,r}^{< \tm -2} \subseteq \cK^{\rm diag}_{s,r}\,.
\]
Indeed, any monomial in $\cK^{out}_{s,r}$ is of the form $x^{P + e_k} \partial_{x_k}$ with $P_k = -1$. Thus its degree is $\|P + e_k\|_{\ell_1} - 1 = \|P\|_{\ell_1} - 2$. In other words, resonant terms which are not divisible by monomials $x^q$, $q\in\cM_{\lambda}$, are of order $\geq \tm - 2$.\\
\noindent
\item Recalling formula \eqref{pappa}, one has 
\begin{equation}\label{bad}
\cI^{(1)}_{s,r} \cap \cK_{s,r}^{\rm diag} = \overline{\mbox{Span}}_\C( x^Q x_k\frac{\partial}{\partial x_k})_{Q\in \cG_\lambda,k\in I}\,.
\end{equation} 
	\end{enumerate}
\end{remark}

\begin{lemma}\label{lem:projections}
\begin{enumerate}
\item The action of $L_\lambda$ preserves   the scaling degree, the subspaces $\ozio, \iacopo$ and $\dom $, and the diagonal vector fields.
\item For any $X,Y\in\iacopo$ we have $\prod^{\cI^0}[X,Y] = 0$.
\item For any $X\in\iacopo$ and $Y\in\dom$, then $[X,Y]\in\dom$.
\item For any $X\in\cV_{s,r}$ and $Y\in\dom$, then $\prod^{\cI^0}[X,Y] = 0$.
\end{enumerate}
\end{lemma}

\begin{proof}
	We recall that if $A,B \in \N_\fin^I$,  $R$, $R_i$ (e.g. monomial) vector fields, we have:
	\begin{eqnarray}
		[x^AR_1,x^BR_2] &= & x^{A+B}[R_1,R_2]+x^AR_1(x^B)R_2-x^BR_2(x^A)R_1\label{bravf}\\
		R(x^{A+B})&=&x^AR(x^B)+x^BR(x^A)\label{brafct}
	\end{eqnarray}
	Here,  
	$R(x^A)=\sum_{i\in I}R_i\frac{\partial x^A}{\partial x_i}$ 
	denotes the Lie derivative of $x^A$ along $R$.
\begin{enumerate}
	\item  It follows directly from \eqref{diago}. 
\item By definition, $X$ (resp. $Y$) is a sum of vector fields of the form $x^AR_1$ with $A\in \cJ^{(1)}\setminus \cJ^{(2)}$. According to (\ref{bravf}), $[X,Y]$ is a sum of vector fields of the form $x^C R_3$ with $C\in \cJ^{(1)}$.

\item If $B\in \cJ^{(1)}\setminus \cJ^{(2)}$ and $A\in \cJ^{(1)}\setminus \cJ^{(1)}\setminus  \cJ^{(2)}$ then, according to \ref{brafct}, $x^AR_1(x^B)\in \cJ^{(2)}$ for any vector field $R_1$ and obviously also for $x^{A+B}$ and $x^B$. 
\item If $B\in \cJ^{(1)}\setminus \cJ^{(2)}$ then, according to \ref{brafct}, $x^AR_1(x^B), X^AR_1(X^B), X^BR_2(x^A)\in \cJ^{(1)}$ for all multi-indices A with non-negative entry and all vector fields $R_1, R_2$.

\end{enumerate}
\end{proof}

We remark that the projection $\Pi^{\cI^{(0)}}$ on diagonal vector fields can be expressed as
\begin{equation}
\Pi^{\cI^{(0)}} \sum_{k\in I} \cY^{(k)}(x) x_k \base{x_k} = \sum_{q + e_k\in \cJ^{(0)}}\cY^{(k)}_qx^q x_k \base{x_k} =:   \sum_{k\in I} \cT_k^{(0)}[\cY^{(k)}(x)] x_k \base{x_k}\,,
\end{equation}
where $\cT_k^{(0)}[f] := \sum_{q + e_k\in\cJ^{(0)}} f_q x^q$.
%
\\


  \begin{remark}\label{mandrione}
  	If we consider a ``diagonal" vector field  $\sum_{k\in I} \cY^{(k)}(x) x_k \base{x_k}$ in the range $\cR$ of $L_\lambda$,  then 
  	\begin{equation}
  		\label{azione ad -1}
  		L_\lambda^{-1} Y = \sum_{k\in I}(L_\lambda^{-1}\cY^{(k)}(x)){x_k}\base{x_k}  \quad \text{where}\quad L_\lambda^{-1}f(x) = \sum_{\substack{q\in \N^I_\fin \\ \lambda\cdot q \neq 0}} \frac{f_q}{\lambda\cdot q} x^q \quad \forall f\in\cH_{s,r} \,.
  	\end{equation}
  \end{remark}

In order to prove Proposition \ref{homo estesa}, we need to show  the invertibilty of  the operators $\Pi^{\cI^{(i)}} L_{ \tD(\lambda) + Z} \Pi^{\cI^{(i)}}$.

To this purpose let us set, for $i=0,1$
\[
A_i= \Pi^{\cI^{(i)}} L_\lambda \Pi^{\cI^{(i)}}\,,\quad B_i= \Pi^{\cI^{(i)}} L_{ Z} \Pi^{\cI^{(i)}}.
\]

\begin{lemma}
	The operator $A^{-1}_iB_i: \mathcal{A}^{\ge \tM_*}_{s,r}\to \cI^{(i)} \cap  \mathcal{A}^{\ge \tM_*}_{s+\sigma, r - \rho} $ is nilpotent of order two.
\end{lemma}
\begin{proof}
First note that,  since $Z$ has order $\geq 1$,  then the operator $B_i$ maps vector fields of order $\geq \tM_*$ to vector fields of order $\geq \tM_*$ and that $A_i$ is invertible on the image of $B_i$ with $A^{-1}_i =  \Pi^{\cI^{(i)}} \ad^{-1}_{ \tD(\lambda)} \Pi^{\cI^{(i)}}$ according to relations \eqref{condiz range}. 

We recall that $Z\in \cK^{\rm diag}\cap \mathcal{A}_{s,r}^{\ge 1}$ implies that $Z\in \iacopo\oplus\dom$. Then we write $Z= Z_1+Z_2$ where 
$Z_j\in \cI^{(j)}$, for $j=1,2$ and note that $B_i= \Pi^{\cI^{(i)}} \ad_{Z_1} \Pi^{\cI^{(i)}}$ , by Lemma \ref{lem:projections} (4)-(3). Moreover by \eqref{bad} we may write
\[
Z_1= \sum_{k\in I, Q\in \cG_\lambda} Z_{Q+e_k}^{(k)} x^Q x_k \base{x_k}\,.
\]
	Let us first consider the case of $\cI^{(0)}$. 
W. l.o.g. we assume $U\in \cI^{(0)}\cap \mathcal{A}^{\ge \tM_*}_{s,r}$ . We start by noticing that
\begin{equation}\label{B azione}
\begin{aligned}
B_0 U &= \Pi^{\cI^{(0)}} \sum_{k\in I, Q\in \cG_\lambda} Z_{Q+e_k}^{(k)} [x^Q  x_k \frac{\partial}{\partial{x_k}}, U ]
\\
&= \Pi^{\cI^{(0)}}\sum_{k\in I, Q\in \cG_\lambda} Z_{Q+e_k}^{(k)} \Big(x^Q  x_k \sum_j\frac{{\partial U}^{(j)}}{\partial{x_k}}\base{x_j}  -(L_U x^Q  x_k) \frac{\partial}{\partial{x_k}}\Big)
\\
&= -\Pi^{\cI^{(0)}}\sum_{k\in I, Q\in \cG_\lambda} Z_{Q+e_k}^{(k)}  (L_U x^Q ) x_k \frac{\partial}{\partial{x_k}}\,.
\end{aligned}
\end{equation}
Note that the first summand in the parenthesis of the second line above has $0$ projection on $\cI^{(0)}$ because of the $x^Q$ factor. Similarly for $x^QL_U x_k$. \\
 Since $A_0 = \Pi^{\cI^{(0)}} \ad_{ \tD(\lambda)} \Pi^{\cI^{(0)}}$ is invertible and preserves degree, and recalling that the projections are idempotent, that is $P^2=P$, we have $A^{-1}_0\Pi^{\cI^{(0)}} = \Pi^{\cI^{(0)}} A^{-1}_0 = A^{-1}_0$.  Consequently, let $V_0:= A^{-1}_0B_0U$. By analyticity it admits the Taylor expansion $\sum_{k\in I, q\in\N^I_\fin} V^{(k)}_{0,q}x^q\base{k}$. Moreover, by construction

\begin{align*}
V_0 &= -\sum_{k\in I, Q\in \cG_\lambda} Z_{Q+e_k}^{(k)} A_0^{-1} \Pi^{\cI^{(0)}}(( L_U x^Q ) x_k \frac{\partial}{\partial{x_k}})\\
&=-\sum_{k\in I, Q\in \cG_\lambda} Z_{Q+e_k}^{(k)} \Pi^{\cI^{(0)}} L_\lambda^{-1}  ( (L_U x^Q ) x_k \frac{\partial}{\partial{x_k}})\,\\
& = \sum_{\substack{k\in I, p\in\N^{I}_\fin\\ p + e_k \in \cJ^{(0)}}} V^{(k)}_{0,p + e_k} x^p x_k\base{x_k}\,,
\end{align*}
where the last equality follows from the fact that ${L^{-1}_\lambda}$ preserves diagonal vector fields. 
Note that $V_0\in\cI^{(0)}$, hence from formula \eqref{B azione} it follows that

\begin{align*}
B_0 A^{-1}_0 B_0 U &= B_0 V = -\Pi^{\cI^{(0)}}\sum_{k\in I, Q\in \cG_\lambda} Z_{Q+e_k}^{(k)}  (L_{V_0} x^Q ) x_k \frac{\partial}{\partial{x_k}}
\\
& =  \Pi^{\cI^{(0)}}\sum_{k\in I, Q\in \cG_\lambda} Z_{Q+e_k}^{(k)}  \Big(\sum_{h\in I, p\in\N^{I}_\fin} V^{(h)}_{0,p + e_h} x^p x_h \frac{\partial x^Q }{\partial x_h} \Big)x_k \frac{\partial}{\partial{x_k}}\,\\
& =  \Pi^{\cI^{(0)}}\sum_{k\in I, Q\in \cG_\lambda} Z_{Q+e_k}^{(k)}  \Big(\sum_{h\in I, p\in\N^{I}_\fin} V^{(h)}_{0,p + e_h} x^{p } x^{Q} Q_h\Big)x_k \frac{\partial}{\partial{x_k}} = 0\,. 
\end{align*}

The case of $B_1= \Pi^{\cI^{(1)}} \ad_{ Z_1} \Pi^{\cI^{(1)}}$, follows almost verbatim from the discussion about $\cI^{(0)}$, by replacing accordingly $\Pi^{\cI^{(0)}} $ with $\Pi^{\cI^{(1)}}$ . In fact, consider formula \eqref{B azione} replacing $\Pi^{\cI^{(0)}}$  with $\Pi^{\cI^{(1)}}$ and note that, taken w.l.o.g $U\in \cI^{(1)}$, then for each monomial of $U$, we have $$x^Q x_k\frac{\partial}{\partial x_k} U^{(j)}_q x^q = x^Q  x^q q_k U^{(j)}_q\,\quad q \in \jacopo\setminus\jdom\,.$$
Recalling that $q\in \jacopo\setminus\jdom$ implies that $q= P + q'$ with $P\in \cG_\lambda$ and  $q'\in \N^I_\fin$, we have that   the first summand of the second line of \eqref{B azione} belongs to $\dom$. For the same reason $L_U(x^Qx_k)$ reduces to $(L_Ux^Q)x_k$. In conclusion we have
$$ 
B_1U = -\Pi^{\cI^{(1)}}\sum_{k\in I, Q\in \cG_\lambda} Z_{Q+e_k}^{(k)}  (L_U x^Q ) x_k \frac{\partial}{\partial{x_k}}\,.
$$
Then
 \[
V_1:= A^{-1}_1B_1U
= \sum_{k\in I, p + e_k\in\jacopo\setminus\jdom} V^{(k)}_{p + e_k} x^p x_k\base{x_k}\,, 
 \]
 and 
 consequently
 \[
 B_1 A^{-1}_1 B_1 U = \Pi^{\cI^{(1)}}\sum_{k\in I, Q\in \cG_\lambda} Z_{Q+e_k}^{(k)}  \Big(\sum_{h\in I, p+e_h \in \jacopo\setminus\jdom} V^{(h)}_{p + e_h} x^{p } x^{Q} Q_h\Big)x_k \frac{\partial}{\partial{x_k}} = 0\,. 
 \]
\end{proof}
The above lemma implies that $A+B$ is invertible on $\mathcal{A}^{\ge \tM_*}_{s,r} \cap (\ozio {\oplus} \iacopo)$ and
\begin{equation}\label{inversaaa}
(A+B)^{-1} =(\id + A^{-1}B)^{-1} A^{-1} = (\id - A^{-1}B ) A^{-1}= A^{-1} - A^{-1} B A^{-1}\,.
\end{equation}

We are now ready to prove Proposition \ref{homo prop}.

\begin{proof}[Proof of Proposition \ref{homo prop}]
Let us start with the case $\ozio$.  By identity \eqref{inversaaa},  we have that 
$$
F = A^{-1}Y - A^{-1}BA^{-1} Y\,.
$$
By Proposition \ref{adjoint action} and Proposition \ref{fan}, we have 
\begin{equation}
\begin{aligned}
\|F\|_{s + \sigma, r- \rho} &\lesssim \g^{-1} {e^{2^6\tc/\s^6}} ( \|Y\|_{s, r} + \|BA^{-1}Y\|_{s + \frac{\s}{2}, r - \rho} )\\
& \lesssim \g^{-1} {e^{2^6\tc/\s^6}} \left( \|Y\|_{s, r} + 4\left(1+\frac{r}{\rho}\right) \|Z\|_{s+\frac{\s}{2},r} \|A^{-1}Y\|_{s + \frac{\s}{2}, r}\right)\\
& \lesssim \g^{-1} {e^{2^6\tc/\s^6}} \left( \|Y\|_{s, r} + 4\left(1+\frac{r}{\rho}\right)\g^{-1} {e^{2^6\tc/\s^6}}  \|Z\|_{s+\frac{\s}{2},r} \|Y\|_{s , r}\right)\,.
\end{aligned}
\end{equation}
{As $\frac{r}{\rho}\geq  1$, then $4\left(1+\frac{r}{\rho}\right)e^{2^6\tc/\s^6}\leq 8\frac{r}{\rho}e^{2^6\tc/\s^6}$ we obtain
$$
\|F\|_{s + \sigma, r- \rho} \lesssim 8\g^{-1}e^{2^7\tc/\s^6}\frac{r}{\rho}\left( 1 + \g^{-1} \|Z\|_{s,r} \right)\|Y\|_{s, r}\,.
$$}
The bound follows. The case $\iacopo$ follows verbatim.
\end{proof}

\subsection{KAM algorithm}

The proof of Theorem \ref{main-thm} follows  directly from a KAM iteration which in turn is based on the repeated application of the following procedure. \\
In line with the decomposition of $\cV_{s,r}$ as a direct sum of the $\cI^{(i)}\, i=0,1,2$, in the following it will be convenient to use the following slightly different but equivalent norm, 

\begin{equation}	\label{norma split}
\bnorm{X}_{s,r} := \max_{0\le j \le 2} \{\|X_j\|_{s,r}\},
\end{equation}
where $X_j$ is the projection of $X$ on $\cI^{(j)}_{s,r}$, for $j = 0,1,2\,.$
\\

Since $\|X_j\|_{s,r} \le \|X\|_{s,r}$ (recall we are using majorant like norms) the norm defined in \eqref{norma split} satisfies
$$
\bnorm{X}_{s,r} \leq \norma{X}_{s,r} \leq 3\, \bnorm{X}_{s,r}\,.
$$

\noindent
\textbf{Main KAM step.} Let $W \in \cV_{\mathtt s,\mathtt R}^{\ge 0}$ 
be of the form  $$W = \tD(\lambda) + Z + X + N, $$ 
with $Z \in \cK^{diag}_{s,r}\cap \cV^{(1\le \td \le \tM_* - 1)}_{s,r} $, $X \in \cV_{ s,r}^{\ge \tM_*} \cap (\cI^{(0)} \cup \cI^{(1)}),\, N\in \cV_{s,r}^{\ge \tM_*} \cap \cI^{(2)}\,. $ We have the following

\begin{lemma}[Main step]\label{uuu}
Given $\g>0, \rho < \frac{r}{5}, \s>0$, assume that 
\begin{equation}\label{assunzione}
\pa{1 + \frac{\bnorm{Z}_{s,r}}{\g} +  \frac{\bnorm{N}_{s,r}}{\g}}^3\frac{\bnorm{X}_{s,r}}{\g} \le \tK_1  \frac{\rho^4 }{r^4} e^{-\frac{2^8\tc}{\s^6}}
\end{equation}
where $\tK_1$ is a pure positive constant. Then, there exists a generating vector field $$F\in \cV_{ s + 2\s,r - 3\rho}^{\ge \tM_*} \cap (\cI^{(0)} \cup \cI^{(1)})$$ satisfying 
\begin{equation}\label{x hamflow}
\bnorm{F}_{s + 2\s, r - 3\rho} \le \frac{\rho}{8 e (r - 3\rho)}\,,
\end{equation}
such that  for all $s_1\ge s+2\s$  the time $1$-flow 
$\Phi_F: B_{r-5\rho}(\tg_{s_1})\to
B_{r -3\rho}(\tg_{s_1})$   is well defined, analytic, symplectic with the bounds
\begin{equation}\label{pedicini}
\sup_{u\in  B_{r-5\rho}(\tg_{s_1})} 	\norm{\Phi^1_F(u)-u}_{s_1}
\le
(r+\rho)  \bnorm{F}_{s, r-3\rho }\,,
\end{equation}
and such that  
$$
W_+ := \exp(L_F) W = \tD(\lambda) + Z + X_+ + N_+\,,
$$
with 
$$
X_+ \in \cV_{ s + 2\s,r - 5\rho}^{\ge \tM_*} \cap (\cI^{(0)} \cup \cI^{(1)}),\,\quad N_+\in \cV_{s + 2\s,r - 5\rho}^{\ge \tM_*} \cap \cI^{(2)}\,. 
$$
More specifically, the following bounds hold:

\begin{align}\label{assassino}
		\bnorm{F}_{s + 2\s, r - 3\rho} & \lesssim \pa{\frac{r}{\rho}}^3  e^{2^8\tc/\s^6}\frac{\bnorm{X}_{s,r}}{\g}\left(1 + \frac{\bnorm{Z}_{s,r}}{\g} + \frac{\bnorm{N}_{s,r}}{\g}\right)^3 \,, \\
\bnorm{X_+}_{s + 2\s,r - 5\rho} &\lesssim \pa{\frac{r}{\rho}}^8 \g^{-1} e^{{2^9}\tc/\s^6}{\pa{1 + \frac{\bnorm{Z}_{s,r}}{\g} + \frac{\bnorm{N}_{s,r}}{\g}}^7} \bnorm{X}_{s,r} ^2\\
\bnorm{N - N_+}_{s + 2\s,r - 5\rho} & \lesssim \pa{\frac{r}{\rho}}^8 e^{{2^9}\tc/\s^6}{\pa{1 + \frac{\bnorm{Z}_{s,r}}{\g} + \frac{\bnorm{N}_{s,r}}{\g}}^7}\pa{ \frac{\bnorm{Z}_{s,r}}{\g} + \frac{\bnorm{N}_{s,r}}{\g} + \frac{\bnorm{X}_{s,r}}{\g}} \bnorm{X}_{s,r}\,.
\end{align}

\end{lemma}

This process will be proven to converge and will yield an analytic vector field $\Phi^{\infty}_*W= D_{\lambda}+Z+N^{\infty}$ in some open ball at the origin.\\
{In order to proceed with the main step}, we shall construct the desired diffeomorphism as the exponential $\exp {L_F}$, where $F = F_0 + F_1$ with $F_i \in \cI^{(i)},\, i= 0,1$.  \\ Note that,  by construction the addenda $X_0$ and $X_1$ belong to $\cR^{\ge \tM^*}_{s,r}$,  for all $r \le \mathtt R$.  \\
Let us expand 
$$\exp{L_F} W =  \tD(\lambda) + Z + X_0 + X_1 + X_2 + [F,  \tD(\lambda) + Z + X_2 ] + [F, X_0 + X_1] + \sum_{k\ge 2} \frac{L^k_F W}{k!}. $$

We shall fix $F_0$ and $F_1$ as the (unique!) solutions of the homological equations

\begin{equation}\label{homosys}
\begin{aligned}
\Pi^{\cI^{(0)}} [F,  \tD(\lambda) + Z + N] &= - X_0 \\
\Pi^{\cI^{(1)}} [F,  \tD(\lambda) + Z + N ] &= - X_1
\end{aligned}
\end{equation}

Since $Z$ it is a diagonal vector field of degree $\ge 1$, then it necessarily belongs to $\iacopo\oplus\dom$. Moreover, by  Lemma \ref{lem:projections} (2)-(4), the first equation reduces to 
\begin{equation}\label{0 homo}
\Pi^{\cI^{(0)}} [F_0, \tD(\lambda) + Z] = - X_0 
\end{equation}
similarly the second equation reduces to 
\begin{equation}\label{1 homo}
\Pi^{\cI^{(1)}} ([F_1, \tD(\lambda) + Z] + [F_0,N])= - X_1\,. 
\end{equation}
The system of equations \eqref{0 homo}-\eqref{1 homo} is triangular and admits a unique solution.\\
Let us start with equation \eqref{0 homo}. By Proposition \ref{homo prop} the unique solution $$F_0 = (\Pi^{\ozio} L_{\lambda}\Pi^{\ozio} + \Pi^{\ozio} L_{Z}\Pi^{\ozio})^{-1}   (- X_0) $$   satisfies

\begin{equation*}
\|F_0\|_{s+\s,r - \rho} \lesssim \frac{r}{\rho} \g^{-1} {e^{2^7\tc/\s^6}} (1 + \g^{-1}\|Z\|_{s,r}) \|X_0\|_{s,r}\,.
\end{equation*}
Plugging it into equation \eqref{1 homo}, we determine analogously $F_1$ which, by Proposition \ref{fan} (recall that $\frac{r}{\rho}>1$), satisfies
\begin{align*}
\|F_1\|_{s + 2\s, r - 3\rho} &\lesssim \frac{r}{\rho} \g^{-1} {e^{2^7\tc/\s^6}} (1 + \g^{-1}\|Z\|_{s+\s,r - 2\rho}) \|X_1 + [F_0, N]\|_{s+\s,r - 2\rho}\\
& \lesssim \frac{r}{\rho} \g^{-1} {e^{2^7\tc/\s^6}}(1 + \g^{-1}\|Z\|_{s,r}) \left[\|X_1\|_{s+\s,r - \rho} +  4(1 + \frac{r - \rho}{\rho})\|F_0\|_{s+\s, r - \rho}\|N\|_{s+\s, r - \rho} \right]\\
& \lesssim \pa{\frac{r}{\rho}}^3 \g^{-1}  {e^{2^8\tc/\s^6}}(1 + \g^{-1}\|Z\|_{s,r})^2 {\left(\norma{X_0}_{s,r} + \norma{X_1}_{s,r}\right) }\pa{1 + \g^{-1}\|N\|_{s,r}} \\
& \lesssim \pa{\frac{r}{\rho}}^3 \g^{-1}  {e^{2^8\tc/\s^6}}(1 + \g^{-1}\bnorm{Z}_{s,r} + \g^{-1}\bnorm{N}_{s,r})^3 \bnorm{X}_{s,r}\,.
\end{align*}
Hence, 
\begin{equation}\label{F stimafin}
	\bnorm{F}_{s + 2\s, r - 3\rho} \lesssim \pa{\frac{r}{\rho}}^3  e^{2^8\tc/\s^6}\frac{\bnorm{X}_{s,r}}{\g}\left(1 + \frac{\bnorm{Z}_{s,r}}{\g} + \frac{\bnorm{N}_{s,r}}{\g}\right)^3 \,, 
\end{equation}
which, by \eqref{assunzione}, yields estimate \eqref{x hamflow}. 
{Recalling that $L_\lambda = [D_{\lambda}, \cdot]$ preserves monomial vector fields and scaling degree,} we have that 
\begin{align*}
(e^{L_F}) W &=  \tD(\lambda) + Z  + N + \Pi^{(2)}[F,  \tD(\lambda) + Z + N ] + [F, X_0 + X_1] + \sum_{k\ge 2} \frac{\ad^k_F (\tD(\lambda) + Z + N)}{k!} +  \sum_{k\ge 2} \frac{\ad^k_F (X_0 + X_1)}{k!}\\
& =  \tD(\lambda) + Z  + N + \Pi^{(2)}[F,  Z + N ] + \sum_{k\ge 2} \frac{\ad^k_F (\tD(\lambda) + Z + N)}{k!} +  \sum_{k\ge 1} \frac{\ad^k_F (X_0 + X_1)}{k!}\\ 
&\stackrel{\eqref{homosys}}{=}  \tD(\lambda) + Z  + N + \Pi^{(2)}[F,  Z + N ] + \sum_{k\ge 2} \frac{\ad^{k-1 }_F (\Pi^{(2)}[F, Z + N]-X_0-X_1)}{k!} +  \sum_{k\ge 1} \frac{\ad^k_F (X_0 + X_1)}{k!}\\ 
&  = \tD(\lambda) + Z  + N +\sum_{k\ge 1} \frac{\ad^{k-1 }_F \Pi^{(2)}[F, Z + N]}{k!} - \sum_{k\ge 1} \frac{\ad^{k}_F (X_0+X_1)}{k+1!} +  \sum_{k\ge 1} \frac{\ad^k_F (X_0 + X_1)}{k!}\\ 
& =  \tD(\lambda) + Z  + N +\sum_{k\ge 1} \frac{\ad^{k-1 }_F \Pi^{(2)}[F, Z + N]}{k!}  +  \sum_{k\ge 1}\ad^k_F (X_0 + X_1) \frac{k}{(k+1)!}\\
& = \tD(\lambda) + Z  + X^+ + N^+\,,
\end{align*}

where $X^+ = X^+_0 + X^+_1$.

We now systematically make use of Propositions \ref{fan} and \ref{ham flow}. Note that in the first series 
$$
\sum_{k\ge 1} \frac{\ad^{k-1 }_F \Pi^{(2)}[F, Z + N]}{k!}
$$
the term $k=1$ does not contribute to $X^+$ but only to $N^+$. 
 The following estimates hold.
\begin{align*}
	&\bnorm{X^+}_{s + 2\s, r - 5\rho} \lesssim \frac{r- 4\rho}{\rho}\|F\|_{s + 2\s,r-4\rho}\|[F,Z+N]\|_{s + 2\s,r-4\rho} + \frac{r - 4\rho}{\rho} \|F\|_{s + 2\s,r-4\rho} \|{X_0 + X_1}\|_{s + 2\s, r - 4\rho}\\
	& \lesssim  \pa{\frac{r}{\rho}}^2\bnorm{F}^2_{s + 2\s,r-3\rho}\bnorm{Z+N}_{s,r} + \pa{\frac{r}{\rho}}\bnorm{F}_{s + 2\s,r-3\rho}\bnorm{X}_{s,r} \\
	& \lesssim \pa{\frac{r}{\rho}}^8 \g^{-2} e^{{2^9}\tc/\s^6} (1 + \g^{-1}\bnorm{Z}_{s,r}+ \g^{-1}{\bnorm{N}_{s,r}})^6 {\bnorm{X}_{s,r}^2}  \bnorm{Z+N}_{s,r} \\
	&+\pa{\frac{r}{\rho}}^4 \g^{-1} e^{2^8\tc/\s^6} (1 + \g^{-1}\bnorm{Z}_{s,r} + \g^{-1}{\bnorm{ N}_{s,r}})^3{\bnorm{X}^2_{s,r}} \\
	& \lesssim \pa{\frac{r}{\rho}}^8 \g^{-1} e^{{2^9}\tc/\s^6}{\pa{1 + \frac{\bnorm{Z}_{s,r}}{\g} + \frac{\bnorm{N}_{s,r}}{\g}}^7} \bnorm{X}^2_{s,r}\,. 
	\end{align*}
	\begin{align*}
	& \bnorm{N^+ - N}_{s + 2\s, r - 5\rho} \lesssim  \pa{\frac{r}{\rho}}\norma{F}_{s + 2\s,r-3\rho}\|Z+N\|_{s,r} + \pa{\frac{r}{\rho}}^8 \g^{-1} e^{{2^9}\tc/\s^6}{\pa{1 + \frac{\bnorm{Z}_{s,r}}{\g} + \frac{\bnorm{N}_{s,r}}{\g}}^7} \bnorm{X}^2_{s,r} \\
	&\lesssim \pa{\frac{r}{\rho}}^4 e^{2^8\tc/\s^6}  \bnorm{X}_{s,r} \pa{ 1 + \frac{\bnorm{Z}_{s,r}}{\g} + \frac{\bnorm{N}_{s,r}}{\g}}^3 \pa{\frac{\bnorm{Z}_{s,r}}{\g} + \frac{\bnorm{N}_{s,r}}{\g}}+\\
	& +  \pa{\frac{r}{\rho}}^8 \g^{-1} e^{{2^9}\tc/\s^6}{\pa{1 + \frac{\bnorm{Z}_{s,r}}{\g} + \frac{\bnorm{N}_{s,r}}{\g}}^7} \bnorm{X}^2_{s,r}\\
	&\lesssim \pa{\frac{r}{\rho}}^8 e^{2^9\tc/\s^6}  \bnorm{X}_{s,r} \pa{ 1 + \frac{\bnorm{Z}_{s,r}}{\g} + \frac{\bnorm{N}_{s,r}}{\g}}^7 \pa{\frac{\bnorm{Z}_{s,r}}{\g} + \frac{\bnorm{N}_{s,r}}{\g} + \frac{\bnorm{X}_{s,r}}{\g}}\,.
\end{align*}
\textbf{Iterative Lemma.}
Fix
 $r_0 = 2 r', s_0 = \mathtt{s}, \rho = r', \s =  s' - \mathtt{s}$   and let $\{\rho_n\}_{n\in\N}, \{\s_n\}_{n\in\N}$ be  the  summable  sequences:
 	\begin{equation}\label{amaroni}
 	\rho_n= \frac{\rho}{{10}} 2^{-n}\,,\qquad \sigma_0 = \frac{\s}{8}, \quad \s_n = {\frac{9\s}{4\pi^2 n^2}}\quad \forall n\ge 1\,.
 \end{equation} 
   Let us define recursively
   \begin{eqnarray}\label{pesto}
 &&r_{n+1} = r_n - 5\rho_n\ \to \ r_\infty:=r_0-\rho = r'\qquad   {\rm (decr
 easing)} \nonumber\\
 &&s_{n+1} = s_n + 2\sigma_n\ \to \ s_\infty:=
 s_0+\sigma = s'\qquad  {\rm (increasing)} \nonumber  .
\end{eqnarray}

Let  $$W^0 := \tD(\lambda) + Z + X_0  + N_0\,,
 $$
 where
\begin{equation}\label{bisanzio}
X_0\in \cV_{\mathtt s_0,r_0}^{\ge \tM_*} \cap (\cI^{(0)} \cup \cI^{(1)}),  \quad Z \in \cK^{diag}_{s,\mathtt{R}}\cap \cV^{(1\le \td \le \tM_* - 1)}_{s_0,r_0}\,\quad N_0 \in \cV_{\mathtt s_0,r_0}^{\ge \tM_*}\cap \cI^{(2)}  \,.
\end{equation}

 We define 
\begin{equation}\label{vigili}
 \e_0:=\gamma^{-1} \bnorm{X_0}_{s_0,r_0} ,  \quad \Theta_0:=  \gamma^{-1}\pa{\bnorm{Z}_{s_0,r_0} + \bnorm{N_0}_{s_0,r_0}} +\e_0 
 \end{equation}

\begin{lemma}[Iterative step]\label{iterativo}
Let $r_0,s_0, \rho, \s$ be  as above,
  $\rho_n, \s_n, r_n, s_n, $  as in \eqref{amaroni}-\eqref{pesto},
  $W_0,X_0, Z, N_0$ as in \eqref{bisanzio}
  and
$\eps_0,\Theta_0$  as in \eqref{vigili}.

 There exists  a constant $\frak C>1$ large enough
 such that 
if 
\begin{equation}\label{gianna}
\eps_0 	\leq \pa{1 + \Theta_0}^{-{7}} \tK^{{-1}}\,,
\qquad 
\tK
:= \frak C  \sup_n 2^{{9n}}e^{\crac n^{12}} e^{-\chi^n (2-\chi) } \,,
\qquad
\crac := 2^9\pa{\frac{4\pi^2}{9\s}}^{6} \tc
\end{equation}
($\tc$ defined in Lemma \ref{adjoint action})
then we can iteratively construct a sequence of generating vector fields
$F_i  \in \cV_{\mathtt s_{i+1},r_i - 3\rho_i}^{\ge \tM_*} \cap (\cI^{(0)} \cup \cI^{(1)})$ 
such that the following holds, for $n\ge 0$.

\smallskip

$(1)_n$ For all $ i = 0,\ldots,  n -1 $ and any $s\ge s_{i+1}$ the 
time-1 flow
 $\Phi_{F_i}$ generated by $F_i$   satisfies
	\begin{equation}
\sup_{u\in  {\bar B}_{r_{i+1}}(\tg_{s})} \norm{\Phi_{F_i}(u)- u}_{s} \le \rho 2^{-2i-7} \,\label{ln}
 \end{equation}
 Moreover, for $n\ge 1$
	\begin{equation}\label{ucazzo}
	\Psi_n := \Phi_{F_0}\circ\cdots \circ \Phi_{F_{n-1}} 
	\end{equation}
	is a well defined, analytic map ${\bar B}_{r_n }(\tg_{s}) \to {\bar B}_{r_0}(\tg_{s})$ for all $s\ge s_n$ with the bound
	\begin{equation}
	\label{cosi}
	 \sup_{u\in {\bar B}_{r_n}(\tg_{s})}\abs{\Psi_{n}(u) - \Psi_{n-1}(u)}_s  \le  \rho 2^{-2n + 2}.
	\end{equation}

	$(2)_n$ We set for $i=1,\dots,n$ 
	 $$W_i= \exp(L_{F_{i-1}})W_{i-1}.$$	 We have
	\begin{equation}
\label{cioccolato}
W_{i} = \tD(\lambda) + Z + X_{i} + N_i,\qquad X_{i}, \in \cV_{\mathtt s_{i},r_i}^{\ge \tM_*} \cap (\cI^{(0)} \cup \cI^{(1)})\,, N_i\in \cV_{\mathtt s_{i},r_i}^{\ge \tM_*} \cap \cI^{(2)} .
\end{equation}
 Setting for $ i = 0,\ldots, n  $
	\begin{equation}\label{xhx-i}
	 \eps_i:=\gamma^{-1}\bnorm{X_i}_{r_i, s_i},  \quad \Theta_i:=\gamma^{-1}\pa{\bnorm{{Z}}_{r_i,s_i} + \bnorm{{N_i}}_{r_i,s_i}} +\e_i \,,
	\end{equation}
we have
\begin{equation}
 \e_i \leq   \e_0  e^{- \chi^{i}+1} \,, 
\qquad
\chi:=3/2\,,\qquad
\qquad  \Theta_i \leq   \Theta_0 \sum_{j=0}^i 2^{-j}\, \label{en} \,.
\end{equation}
\end{lemma}

\begin{proof}
	We prove it by induction. The case  $n=0$ follows directly since item $(1)$ is empty and item $(2)$ is tautological. Let us now assume the Lemma holds up to $n$ and prove it for $n+1$.  Our purpose is to apply the Main step Lemma \ref{uuu}. Let us start by proving item $(1)_{n+1}$.  By the  smallness hypothesis \eqref{gianna}, choosing $\mathfrak{C} \ge \frac{8 20^4 e}{\mathtt{K}_1}$, condition \eqref{assunzione} is fullfilled. Thus Lemma \ref{uuu} ensures the existence of $F_n, X_{n+1}, N_{n+1}$. The bound \eqref{ln} follows from the smallness hypothesis, the first bound in \eqref{assassino} and \eqref{pedicini}, provided that  $\mathfrak{C}$
	is sufficiently large to control the constant in \eqref{assassino}. The bound \eqref{cosi} follows readily from \eqref{ln}. \\
	Let us now prove item $(2)_{n+1}$.  By the second and third inequalities in \eqref{assassino} we have:
\begin{align}
	\eps_{n+1} & \le \mathtt{K}_2 \pa{\frac{r_n}{\rho_n}}^{8} e^{\tC'n^{12}} \pa{1 + \Theta_n}^7\eps_n^2 \\
	|\Theta_{n+1} - \Theta_n| &\le \mathtt{K}_3 \pa{\frac{r_n}{\rho_n}}^{8} e^{\tC'n^{12}} \pa{1 + \Theta_n}^7\eps_n \Theta_n\,. 
	\end{align}
	
	Then substituting the inductive hypothesis \eqref{en} together with the smallness condition \eqref{gianna} with $\mathfrak{C}$ large enough, we obtain the bounds \eqref{en} for $n+1$. 
\end{proof}

\begin{cor}\label{convergo} The family of maps $\pa{\Psi_n}_{n}$, the families of vector fields $X_n$ and $N_n$ are all Cauchy sequences. As a consequence $\Psi:=\lim_{n\to \infty}\pa{\Psi_n}_{n}$
is well defined as a map from ${\bar B}_{r'}\pa{\tg_{s'}}$
to $ {\bar B}_{2r'}\pa{\tg_{s'}}$, and $\Psi_{*} W^0 = \tD(\lambda) + Z + N_\infty$, where $N_{\infty} = \lim_{n\to+\infty} N_n$ with $N_\infty\in \cV^{\ge \tM_*}_{s',r'}\cap \cI^{(2)}$.
  \end{cor}
  The proof follows directly from \eqref{cosi} and from \eqref{en}.
  
\begin{proof}[Proof of Theorem \ref{main-thm}]  
  Let us verify that the vector field $W$ satisfies the hypothesis of the iterative lemma. To this purpose, let us decompose\footnote{Here the sub-index does not represent the component on the subspace $\cI^{(0)}$.} $X = X_0 + N_0$ where $X_0 \in \cI^{(0)}\oplus \cI^{(1)}$ and $N_0 \in\cI^{(2)}$. Recalling that 
  $$
  \norma{X}_{\mathtt{s},2r'} \leq \pa{\frac{2r'}{\mathtt{R}}}^{\tM^*} \norma{X}_{\mathtt{s},\mathtt{R}}\,, \quad \quad \norma{Z}_{\mathtt{s},2r'} \leq \pa{\frac{2r'}{\mathtt{R}}} \norma{Z}_{\mathtt{s},\mathtt{R}}
  $$
  the smallness conditions are met provided that $r'$ is small enough. The result follows.
\end{proof}

  \appendix
  \section{Properties of regular vector fields and proof of Homological equation}
 \subsection{ Proof of Lemma \ref{monotone}}\label{appendicite}
 The proof is a minor adaptation of similar results for Hamiltonian vector fields.
Given a vector field $V\in \cV_{ s,r}$,  we define  a map
 \[
 B_1(\ell^2(I,\C))\to \ell^2(I,\C) \,,\quad y=\pa{y_k}_{k\in I }\mapsto 
 \pa{Y^{(k)}_{V}(y;r,s)}_{k\in I}
 \]
 by setting
 \begin{equation}\label{giggina}
 	Y^{(k)}_{V}(y;r,s) := \sum_\ast |V_q^{(k)}| c^{(k)}_{r,s}(q) y^{q}
 \end{equation}
where we set
 \begin{equation}
 	\label{persico}
 	c^{(k)}_{r,s}(q):=  r^{|q|-1} \pa{\frac{\jap{k}}{\prod_h\jap{h}^{q_h}}}^{2} e^{-s (\sum_h \jap{h}^\theta q_h -\jap{k}^\theta)}\,.
 \end{equation}
 
 For brevity, let us define
 $$
 \sum_\ast:=\sum_{q\in \Z^I_f, k\in I  \;\fm\cdot q =\fm_k}\,.
 $$

 The vector field $Y_V$ is a majorant  analytic function on $\ell^2$ which has the {\it same norm as $V$}. Since the majorant  analytic functions on a given space have a natural ordering this gives us a natural criterion for immersions, as formalized in the following Lemma.
 \begin{lemma}\label{stantuffo}	 Let
 	$\ri,\rf>0,\,s,s'\geq 0.$ The following properties hold.
 	\begin{enumerate}
 		\item   The norm of $V$ can be expressed as
 		\begin{equation}\label{ypsilon}
 			\norm{V}_{r,s}= 
 			\sup_{|y|_{\ell^2}\le 1}\abs{Y_V(y;r,s)}_{\ell^2}
 		\end{equation}
 		\item  Given 
 		$
 		V\in \cV_{\rf,s'}$ 
 		and $W\in \cV_{\ri,s}\,,
 		$
 		\\	
 		such that for all $q\in \N^I_f$ and all $k\in I$ such that $\fm\cdot q =\fm_k$
 		one has
 		\[
 		|V^{(k)}_{q}| c^{(k)}_{\rf,s'}(q)  
 		\le 
 		c
 		|W^{(k)}_{q}| c^{(k)}_{\ri,s}(q),
 		\]
 		for some $c>0,$
 		then
 		\[
 		\norm{V}_{\rf,s'}
 		\le 
 		c
 		\norm{W}_{\ri,s}\,.
 		\]
 	\end{enumerate}
 \end{lemma}
 \begin{proof}Follows directly from the definition of $|\cdot|$ and by \eqref{giggina}.
 \end{proof}

 In order to prove Lemma \ref{monotone} we need some notations and results proven in \cite{Bourgain:2005} and \cite{Yuan_et_al:2017}.
 \begin{defn}\label{n star}
 	Given a vector $v=\pa{v_h}_{h\in I}\in \N^I_f$ with $|v|\ge 2$  we denote by $\na=\na(v)$ the vector $\pa{\na_l}_{l=1}^N$ (where $N$ is finite)  which is the decreasing rearrangement 
 	of
 	$$
 	\{\N\ni j> 1\;\; \mbox{ repeated}\; \sum_{\s=\pm}v_{j,\s} + v_{-j,\s}\; \mbox{times} \} \cup \{ 1\;\; \mbox{ repeated}\; \sum_{\s=\pm}v_{1,\s} + v_{-1,\s} + v_{0,\s}\; \mbox{times}  \}
 	$$
 \end{defn}
 \begin{remark}
 	A good way of envisioning this list is as follows. Given  an infinite set of variables $\pa{x_i}_{i\in\Z}$ and a vector $v=\pa{v_i}_{i\in \Z}\in \N^\Z_f$ consider the monomial $x^v:= \prod_i x_i^{v_i}$. We can write 
 	\[
 	x^v= \prod_hx_h^{v_h} = x_{h_1} x_{h_2}\cdots x_{h_{|v|}}\,,\quad \mbox{ with}\quad h_i\in I
 	\] 
 	then $\na(v)$ is the decreasing rearrangement of the list $\pa{\jap{h_1},\dots,\jap{h_{|v|}}}$.
 	%
 \end{remark}
Given $q\in\N^I_f$ 
with $ |q|\ge 1$ and $k=(j,\s)\in I$  such that $\fm\cdot q=\fm_k$
from now on we define
$$
\na=\na(q+e_k)\,
\qquad \mbox{and set}\quad
N:=|q|+1 
$$
which is the cardinality of $\na.$ 
We  observe that, $N\ge 2$ and since
\begin{equation}
	\label{moment}
	0= \fm\cdot q -\fm_k=\sum_{i\in \Z} i\pa{q_{i,+}- q_{i,-}} -\s j
\end{equation}
there exists a choice of $\s_i = \pm1, 0$ such that
\begin{equation}\label{pi e cappucci}
	\sum_l \sigma_l\na_l=0.
\end{equation}
with $\sigma_l \neq 0$  if $\na_l \neq 1$.
Hence, 
\begin{equation}\label{eleganza}
	\na_1\le\sum_{l\ge 2}\na_l.
\end{equation}
Indeed, if $\sigma_1 = \pm 1$, the inequality follows directly from \eqref{pi e cappucci}; if $\sigma_1 = 0$, then $\na_1=1$ and consequently $\na_l = 1\, \forall l$. Since
$|v|\ge 2$, the list $\na$ has at least two elements, so the inequality is achieved.
\begin{lemma}\label{constance generalbis}
	Given $q\in\N^I_f$ 
	with $ |q|\ge 1$ and $k=(j,\s)\in I$  such that $\fm\cdot q=\fm_k$ we have
	\begin{equation}\label{yuan 2bis}
		\sum_h \jap{h}^\theta q_h + \jap{k}^\theta = \sum_h \jap{h}^\theta v_h =\sum_{l\ge 1} \na_l^\theta  \ge 2 \na^\theta_1+ (2-2^\teta) {\sum_{l\ge 3} \na_l^\theta} .
	\end{equation}
\end{lemma}
\begin{proof}
	The lemma above was proved in  \cite{Bourgain:2005} for $\theta=\frac12$  and for general {$0<\theta<1$} in \cite{Yuan_et_al:2017}[Lemma 2.1].
	\\
	We start by noticing that if $|q|= 1$ then  $\na$ has cardinality equal to two and \eqref{yuan 2bis} becomes $\na_1+\na_2 \ge 2\na_1$.  Now,  by \eqref{eleganza}, momentum conservation implies that
	$\na_1=\na_2$ and hence \eqref{yuan 2bis}.
	\\
	If  $|q\ge 2$ we write
	\[
	\sum_h\jap{h}^\theta v_h -2\na_1^\theta=\sum_{l\ge 2} \na_l^\theta  - \na_1^\theta \ge  \sum_{l\ge 2} \na_l^\theta  - (\sum_{l\ge 2} \na_l)^\theta\ge  \na_2^\theta +\sum_{l\ge 3} \na_l^\theta -{(\na_2+\sum_{l\ge 3}\na_l)}^\theta
	\]
then the proof follows word by word Lemma A.4 of {\cite{stolo-procesi}}.
\end{proof}
The Lemma proved above, is fundamental in discussing the properties of $\cV_{s,r}$ with $s>0$, indeed it implies
\begin{equation}\label{stima1}
	\sum_h\jap{h}^\theta q_h -\jap{k}^\theta= \sum_h\jap{h}^\theta v_h -2\jap{k}^\theta \ge(2-2^\theta) \pa{\sum_{l\ge 3} \na_l^\theta } 
	\ge  0
\end{equation}
for all $q,k$ satisfying momentum. 

\begin{proof}[Proof of Lemma \ref{monotone}]
	In all that follows we shall use systematically the fact that our vector fields preserve {are momentum preserving},  are zero at the origin so that $|q| \ge 1$.
	\\
	We need to show that setting $s'= s+ \delta$ and $r'<r$
	\begin{equation}\label{gigina}
		\frac{	c^{(k)}_{\rf,s+\delta}(q)}{c^{(k)}_{r,s }(q)} =  \pa{\frac{\rf}{r}}^{|q|-1}e^{-\delta (\sum_h\jap{h}^\theta q_h -\jap{k}^\theta)}\le 1\,,
	\end{equation}
	which follows directly from $|q|\ge 1$ and from   \eqref{stima1} of Lemma \ref{constance generalbis} .
\end{proof}
\subsection{Homological Equation}

  	\begin{lemma}\label{mandrione1}
  		For any $p\in\Z^I_\tf$ if
  		\begin{equation}\label{bigliettino}
  			\sum_{k} p_k\jap{k}^\al e^{\im \varphi_{k}} \geq 2 \sum_k |p_k|\,,
  		\end{equation}
  		then 
  		\begin{equation}\label{birrette}
  			|\lambda\cdot p| \geq 1\,.
  		\end{equation}
  	\end{lemma}
  	\begin{proof}
  		By condition \eqref{qalpa} and triangular inequality, the following bounds hold
  		\begin{equation}\label{devoscrivere}
  			|\sum_k p_k\lambda_k| \ge | |\sum_k p_k \jap{k}^\al e^{\im \varphi_{k}}|- |\sum_k p_k (\lambda_k - \jap{k}^\al e^{\im \varphi_{k}})|| \ge \frac32 \sum_k |p_k|\,.
  		\end{equation}
  \end{proof}
  
 \subsubsection*{Proof of Proposition \ref{adjoint action}}
 By Lemma \ref{stantuffo}, it is sufficient to show that for $\lambda$ $(\g,\tau)$-diophantine modulo $\Delta_\lambda$, for all $q,k$ such that, {$|q|\geq 2$,} $\fm\cdot q=\fm_k$  and $\lambda\cdot q-\lambda_k\ne 0$ we have
\begin{equation}
	 \label{scrivomale}
 \frac{	c^{(k)}_{r,s+\delta}(q)}{c^{(k)}_{r,s }(q)}\frac{1}{|\lambda\cdot q-\lambda_k|}= \frac{e^{-\delta (\sum_h\jap{h}^\theta q_h -\jap{k}^\theta)}}{|\lambda\cdot q-\lambda_k|}\le Ce^{\frac{\tc}{\delta^6}}
\end{equation}
 We divide the proof in various cases.
 \\
 {\bf Case 0} If $\na_1=1$ then $q+\be_k$ is supported only on the modes $k=(j,\s)$ with $j=\pm 1,0$. Thus
 \[
 \frac{	c^{(k)}_{r,s+\delta}(q)}{c^{(k)}_{r,s }(q)}\frac{1}{|\lambda\cdot q-\lambda_k|}= \g^{-1}e^{-\delta (|q|-1)}\prod_{h}(1+q_h^2)^\tau=  \g^{-1}e^{-\delta |q|/2}|q|^{12\tau}\,.
 \]
 \\
 {\bf Case 1} If   $q_{k}\ne 0$, then we define $a= q- e_k$, and note that  $a\in \N^I_\tf$ and  \eqref{scrivomale} reads
 \[
\frac{	c^{(k)}_{r,s+\delta}(q)}{c^{(k)}_{r,s }(q)}\frac{1}{|\lambda\cdot q-\lambda_k|}= \g^{-1}e^{-\delta \sum_h\jap{h}^\theta a_h }\prod_h(1+\jap{h^2}a_h^2)^\tau= \g^{-1} e^{\sum_h \tf_h(\theta,a_h)}
 \]
where 
\[\
\tf_h(t,x)=  -\delta \jap{h}^t x + \tau \ln(1+\jap{h}^2 x^2)\,,
\]
then the result follows by {\cite{stolo-procesi}[Lemma A 11]} with $\theta\rightsquigarrow \theta/2$.
\\
 {\bf Case 2} If  $q_k=0$  and $|\lambda\cdot q-\lambda_k||\ge 1/2$, then, using \eqref{stima1},  we have that  \eqref{scrivomale} is bounded by $4$
 \\
 {\bf Case 3} If $q_k=0$,  $|\lambda\cdot q-\lambda_k||< 1/2$, then recalling  \eqref{devoscrivere}
and setting $\lambda^{(0)}_k = \jap{k}^2 e^{\im \varphi_k}$,  we have
 \[
 |\lambda^{(0)}\cdot q-\lambda^{(0)}_k| \le 2(|q|+1)\,,\quad |q|=\sum_k q_k\,.
 \]
 by definition of the $(\na_i)_{i=1}^N$ this means that there is a corresponding sequence $(s_i)_{i=1}^N$ of complex numbers $|s_i|=1$ such that
 \[
| \sum_{i=1}^N s_i \na_i^2|< 2N
 \]
 (recall that $N=|q|+1\ge 3$).  Without loss of generality we may assume that $s_1=1$. Now we have two possibilities:
\\
  {\bf Case 3A}  If $\na_1\ne \na_2$ then 

 \begin{equation}\label{dajeunnome}
 	 \na_1 +\na_2 \le  \na_1^2 -\na_2^2\le  | \na_1^2 +s_2 \na_2^2|\le  2N + \sum_{i=3}^N  \na_i^2 
   \le 7  \sum_{i=3}^N  \na_i^2\,.
 \end{equation}

  \begin{align*}
   \sum_h\jap{h}^{\theta/2} q_h +\jap{k}^{\theta/2} &\le 2\na_1^{\theta/2} + \sum_{i=3}^N  \na_i^{\theta/2}
   \le 2(7  \sum_{i=3}^N  \na_i^2)^{\theta/2} + \sum_{i=3}^N  \na_i^{\theta/2}\\
  & \le  (2 \cdot 7^{\theta/2}+1) \sum_{i=3}^N  \na_i^{\theta}\le\frac{2 \cdot 7^{\theta/2}+1}{2-2^\theta}\sum_h\jap{h}^\theta q_h -\jap{k}^\theta\,.
  \end{align*}
 Thus setting $b= q+e_k$  (and using that $q_h-\delta_{kh} \le b_h$)
 \begin{align*}
 \frac{	c^{(k)}_{r,s+\delta}(q)}{c^{(k)}_{r,s }(q)}\frac{1}{|\lambda\cdot q-\lambda_k|}&\le \g^{-1}e^{-\delta (\sum_h\jap{h}^\theta q_h-\jap{k}^\theta ) }\prod_h(1+\jap{h^2}(q_h-\delta_{kh})^2)^\tau\\
 & \le  \g^{-1}e^{-\delta \tc  \sum_h\jap{h}^{\theta/2} b_h }\prod_h(1+\jap{h^2}b_h^2)^\tau  = \g^{-1} e^{\sum_h \tf_h(\nicefrac\theta 2,b_h)}
 \end{align*}
  the result follows by {\cite{stolo-procesi}[Lemma A 11]}.

   {\bf Case 3B}  If $\na_1=\na_2>\jap{k}$ or $\na_1=\na_2=\na_3=\jap{k}$ then (we may assume that $\na_1>1$ since otherwise we are in case 0)
  \[
  \sum_h\jap{h}^\theta q_h -\jap{k}^\theta\ge \frac13 \sum_h\jap{h}^\theta q_h \,,
  \]
  so that
 \begin{align*}
 	\frac{	c^{(k)}_{r,s+\delta}(q)}{c^{(k)}_{r,s }(q)}\frac{1}{|\lambda\cdot q-\lambda_k|}&\le \g^{-1}e^{-\delta (\sum_h\jap{h}^\theta q_h-\jap{k}^\theta ) }(1+\jap{k}^2)\prod_{h\ne k}(1+\jap{h^2}q_h^2)^\tau\\
 	& \le  \g^{-1}e^{-\delta/2 \sum_h\jap{h}^\theta q_h }\prod_h(1+\jap{h^2}q_h^2)^{2\tau}  
 \end{align*}
 then the result follows from  {\cite{stolo-procesi}[Lemma A 11]} with $\theta \rightsquigarrow \theta/2, \delta/2\rightsquigarrow \delta$ and $2\tau\rightsquigarrow \tau$.
 
 {\bf Case 3C}   If $q_k=0$,  $|\lambda\cdot q-\lambda_k| < 1/2$, $\na_1=\na_2=\jap{k}$ and $\na_3<\na_1$ then there exists one and only one $k_1$ such that $k_1\ne k$,  $\jap{k_1}=\jap{k}$ and for which $q_{k_1}=1$ (all other $h$ such that  $\jap{h}=\jap{k}$ must have $q_{h}=0$).
 Thus the right most inequality in formula \eqref{dajeunnome} reads
 \[
 |\lambda_{k_1}^{(0)}- \lambda_k^{(0)}|= \na_1^2|e^{\im \varphi_{k_1}}- e^{\im \varphi_k}|  \le 2N- \sum_{i=3}^N \na_i^2
 \]
  while, setting $k=(j,\s)$ and $k_1=(j_1,\s_1)$, the momentum conservation reads
  \[
  |\s_1j_1-\s j | \le \sum_{i=3}^N \na_i
  \]
  If $\s_1j_1\ne \s j$ then $\na_1 \le  \sum_{i=3}^N \na_i$, so that
   \begin{align*}
  	\sum_h\jap{h}^{\theta} q_h +\jap{k}^{\theta} &\le 2\na_1^{\theta} + \sum_{i=3}^N  \na_i^{\theta}
  	\le 2(  \sum_{i=3}^N  \na_i)^{\theta} + \sum_{i=3}^N  \na_i^{\theta}\\
  	& \le  3\sum_{i=3}^N  \na_i^{\theta}\le\frac{3}{2-2^\theta}\sum_h\jap{h}^\theta q_h -\jap{k}^\theta\,\,,
  \end{align*}
  then one proceeds as in Case 3A.
  \\
  If $\s_1j_1= \s j$ then, since $k_1\ne k$, one must have $\s_1=-\s$ and $j_1=-j$. Thus, by Assumption \ref{assumption2} either $\lambda_{k_1}=\lambda_k$ or $|e^{\im \varphi_h}- e^{\im \varphi_k}|\ge C$. If $\lambda_{k_1}=\lambda_k$ then
  \[
  \lambda\cdot(q-\be_k)= \lambda \cdot (q- \be_{k_1})\,,\quad  \sum_{h}\jap{h}^\theta q_{h}-\jap{k}^\theta= \sum_h\jap{h}^\theta q_h-\jap{k_1}^\theta
  \] since now $q_{k_1}\ne 0$ we fall in Case 1.
  
  On the other hand if $|e^{\im \varphi_h}- e^{\im \varphi_k}|\ge C$ then
  \[
  \na_1 \le \sqrt{\frac{7}{C}} \sum_{i=3}^N \na_i\,\,
  \]
  and again we proceed as in Case 3A.

  \vspace{0.5cm}
  \gr{Acknowledgements.} 
  J.E. Massetti and M. Procesi have been supported by 
  the  research project 
  PRIN 2022FPZEES ``Stability in Hamiltonian dynamics and beyond"
  of the 
  Italian Ministry of Education and Research (MIUR).
J.E.M. acknowledges also
the support of the Department of Excellence grant MatMod@TOV
(2023-27), awarded to the Department of Mathematics at University of Rome Tor Vergata, the support of the project ``Stable and unstable phenomena in propagation of Waves in dispersive media" of  INdAM-GNAMPA and the moral one of E. Antonelli, L. Baroni, and R. Feola.

  \vspace{0.5cm}
  
  \gr{Declarations}. Data sharing is not applicable to this article as no datasets were generated or analyzed during the current study.
  
  \noindent
  Conflicts of interest: The authors have no conflict of interests to declare.
\bibliographystyle{alpha}
\bibliography{biblioAlmostSob}
\end{document}